\documentclass[11pt]{article}

\usepackage{latexsym}
\usepackage{amssymb}
\usepackage{amsthm}
\usepackage{amscd}
\usepackage{amsmath}
\usepackage{mathrsfs}
\usepackage{graphicx}
\usepackage{hyperref}
\usepackage{tabmac}
\usepackage{shuffle}

\usepackage[all]{xy}
\input xy \xyoption{frame}
\xyoption{dvips}

\usepackage[colorinlistoftodos]{todonotes}
\usepackage{fullpage}
\usetikzlibrary{chains,scopes}

\theoremstyle{definition}
\newtheorem* {theorem*}{Theorem}
\newtheorem* {conjecture*}{Conjecture}
\newtheorem{theorem}{Theorem}[section]

\theoremstyle{definition}

\newtheorem* {example*}{Example}

\newtheorem{lemma}[theorem]{Lemma}
\theoremstyle{definition}
\newtheorem{definition}[theorem]{Definition}
\theoremstyle{definition}

\newtheorem{conjecture}[theorem]{Conjecture}
\newtheorem{proposition}[theorem]{Proposition}
\newtheorem{corollary}[theorem]{Corollary}
\newtheorem{algorithm}[theorem]{Algorithm}
\newtheorem *{remark}{Remark}
\theoremstyle{definition}
\newtheorem {example}[theorem]{Example}
\theoremstyle{definition}

\theoremstyle{definition}

\theoremstyle{definition}

\xyoption{dvips}

\def\({\left(}
\def\){\right)}
\newcommand{\sP}{\mathscr{P}}

\newcommand{\cR}{\mathcal{R}}

\newcommand{\cC}{\mathcal{C}}
\newcommand{\cD}{\mathcal{D}}

\def\NN{\mathbb{N}}

\def\RR{\mathbb{R}}
\def\ZZ{\mathbb{Z}}
\def\Aut{\mathrm{Aut}}

\def\spanning{\textnormal{-span}}

\newcommand{\cN}{\mathcal{N}}

\def\barr{\begin{array}}
\def\earr{\end{array}}
\def\ba{\begin{aligned}}
\def\ea{\end{aligned}}
\def\be{\begin{equation}}
\def\ee{\end{equation}}

\def\qquand{\qquad\text{and}\qquad}

\def\qquord{\qquad\text{or}\qquad}

\def\I{\mathcal{I}}

\def\cH{\mathcal H}

\def\DesR{\mathrm{Des}_R}
\def\DesL{\mathrm{Des}_L}

\def\hs{\hspace{0.5mm}}

\def\id{\mathrm{id}}
\def\PP{\mathbb{P}}

\def\ben{\begin{enumerate}}
\def\een{\end{enumerate}}

\def\cT{\mathscr{T}}

\def\hs{\hspace{0.5mm}}

\def\Des{\mathrm{Des}}

\def\ellhat{\hat\ell}

\def\x{\textbf{x}}
\def\y{\textbf{y}}

\def\a{\textbf{a}}
\def\b{\textbf{b}}

\newcommand{\xRightarrow}[2][]{\ext@arrow 0359\Rightarrowfill@{#1}{#2}}

\renewcommand{\r}[1]{\textcolor{red}{#1}}

\newcommand{\cA}{\mathcal{A}}
\newcommand{\cB}{\mathcal{B}}

\def\iR{\hat\cR}
\def\iHR{\cH\hat\cR}

\def\F{\mathscr{F}}

\def\arcstart{\ \xy<0cm,-.15cm>\xymatrix@R=.1cm@C=.3cm }
\newcommand{\arcstartc}[1]{\ \xy<0cm,-.15cm>\xymatrix@R=.1cm@C=#1cm}

\def\ellhat{\hat\ell}

\def\s{\mathbf{s}}
\def\t{\mathbf{t}}
\def\r{\mathbf{r}}
\def\customstar{\hs\hat \bullet_{\mathrm{des}}\hs}

\def\sA{\mathscr{A}}
\def\sB{\mathscr{B}}
\def\sR{\mathscr{R}}
\def\sC{\mathscr{C}}

\def\ihB{\underline{\isB}}

\newcommand{\sAq}[1]{\sA_{\bullet,\geq #1}}
\newcommand{\sAp}[1]{\sA_{\leq #1,\bullet}}

\def\sX{\mathscr{X}}
\def\sY{\mathscr{Y}}

\def\u{\mathbf{u}}
\def\v{\mathbf{v}}
\def\isB{\hat{\mathscr{B}}}

\def\atildediagram{
\begin{tikzpicture}[
start chain,
node distance=5mm, 
ch/.style={on chain,inner sep=2pt},
chj/.style={ch,join}
]
\node[chj]{1};
\node[chj]{2};
\node[chj]{3};
\node[chj]{$\cdots$};
\node[chj]{$n$};
\begin{scope}[start chain=br going above]
\chainin(chain-3);
\node[ch,join=with chain-1,join=with chain-5] {0};
\end{scope}
\end{tikzpicture}
}

\def\btildediagram{
\begin{tikzpicture}[
node distance=5mm, 
every node/.style={on chain,join,inner sep=1mm}, 
every join/.style={-}
]
{ [start chain=trunk]
\node {$1$};
\node {$2$};
{ [start branch=up going above] } 
\node {$\cdots$};
\node(a) {$(n-1)$};
\node(b) {$n$};
{ [continue branch=up]
\node [on chain] {$0$};
}
\draw[double] (a) -- (b);
 }
\end{tikzpicture}
}

\def\ctildediagram{
\begin{tikzpicture}[
node distance=5mm, 
every node/.style={on chain,join,inner sep=1mm}, 
every join/.style={-}
]
{ [start chain=trunk]
\node(x) {$0$};
\node(y) {$1$};
\node {$2$};
\node {$\cdots$};
\node(a) {$(n-1)$};
\node(b) {$n$};
\draw[double] (a) -- (b);
\draw[double] (x) -- (y);
 }
\end{tikzpicture}
}

\def\dtildediagram{
\begin{tikzpicture}[
node distance=5mm, 
every node/.style={on chain,join,inner sep=1mm}, 
every join/.style={-}
]
{ [start chain=trunk]
\node {$1$};
\node {$2$};
{ [start branch=up going above] } 
\node {$\cdots$};
\node(a) {$(n-2)$};
{ [start branch=upp going above] } 
\node(b) {$n-1$};
{ [continue branch=up]
\node [on chain] {$0$};
}
{ [continue branch=upp]
\node [on chain] {$n$};
}
 }
\end{tikzpicture}
}

\def\hthreediagram{
\begin{tikzpicture}
\node[inner sep=1mm] (a) at (-2,0) {
  $a$
};
  \node[inner sep=1mm] (b) at (-1,0) {
  $x$
};
  \node[inner sep=1mm] (c) at (0,0) {
  $b$
};
\draw[double] (b) --  node [anchor=south] {5} (c);
\draw
(a)   --  (b) 
;
\end{tikzpicture}
}

\def\athreediagram{
\begin{tikzpicture}
\node[inner sep=1mm] (a) at (-2,0) {
  $a$
};
  \node[inner sep=1mm] (b) at (-1,0) {
  $x$
};
  \node[inner sep=1mm] (c) at (0,0) {
  $b$
};
\draw
(a)   --  (b)
(b) -- (c)
;
\end{tikzpicture}
}

\def\bthreediagram{
\begin{tikzpicture}
\node[inner sep=1mm] (a) at (-2,0) {
  $a$
};
  \node[inner sep=1mm] (b) at (-1,0) {
  $x$
};
  \node[inner sep=1mm] (c) at (0,0) {
  $b$
};
\draw[double] (b) -- node [anchor=south] {4} (c);
\draw
(a)   --  (b) 
;
\end{tikzpicture}
}

\def\hthreediagram{
\begin{tikzpicture}
\node[inner sep=1mm] (a) at (-2,0) {
  $a$
};
  \node[inner sep=1mm] (b) at (-1,0) {
  $x$
};
  \node[inner sep=1mm] (c) at (0,0) {
  $b$
};
\draw[double] (b) --  node [anchor=south] {5} (c);
\draw
(a)   --  (b) 
;
\end{tikzpicture}
}

\def\dfourdiagram{
\begin{tikzpicture}
\node[inner sep=1mm] (a) at (-2,0) {
  $a$
};
  \node[inner sep=1mm] (b) at (-1,0) {
  $x$
};
  \node[inner sep=1mm] (c) at (-1,0.8) {
  $c$
};
  \node[inner sep=1mm] (d) at (0,0) {
  $b$
};
\draw
(a)   --  (b) 
(b) -- (c)
(b) -- (d)
;
\end{tikzpicture}
}

\def\andiagram{
\begin{tikzpicture}[
start chain,
node distance=5mm, 
every node/.style={on chain,join,inner sep=1mm}, 
every join/.style={-}
]
\node {1};
\node  {2};
\node {$\cdots$};
\node  {$n$};
\end{tikzpicture}
}

\def\bndiagram{
\begin{tikzpicture}[
start chain,
node distance=5mm, 
every node/.style={on chain,join,inner sep=1mm}, 
every join/.style={-}
]
\node {1};
\node {$\cdots$};
\node(a)  {$(n-1)$};
\node(b)  {$n$};
\draw[double] (a) -- (b);
\end{tikzpicture}
}

\def\dndiagram{
\begin{tikzpicture}[
node distance=5mm, 
every node/.style={on chain,join,inner sep=1mm}, 
every join/.style={-}
]
{ [start chain=trunk]
\node {1};
\node {$\cdots$};
\node [on chain] {$(n-2)$};
{ [start branch=up going above] } 
\node [on chain] {$(n-1)$};
{ [continue branch=up]
      \node [on chain] {$n$};
    }
 }
\end{tikzpicture}
}

\numberwithin{equation}{section}
\allowdisplaybreaks[1]
\UseCrayolaColors

\makeatletter
\renewcommand{\@makefnmark}{\mbox{\textsuperscript{}}}
\makeatother

\begin{document}
\title{Braid relations for involution words in affine Coxeter groups}

\author{
    Eric Marberg \\
    Department of Mathematics \\
    HKUST \\
    {\tt eric.marberg@gmail.com}
}

\date{}

\maketitle

\begin{abstract}
We describe an algorithm to identify a minimal set of ``braid relations''
which span and preserve all sets of involution words for twisted Coxeter systems of finite or affine type. 
We classify  the cases  in which adding the smallest possible set of ``half-braid'' relations to the ordinary
braid relations produces a spanning set: in the untwisted case, this occurs for the
Coxeter systems  which are finite with rank two or type $A_n$, or affine with rank three or type $\tilde A_n$.
These results generalize recent work of Hu and Zhang on the finite
classical cases.
\end{abstract}

\setcounter{tocdepth}{2}

\section{Introduction}

Let $(W,S)$ be a Coxeter system with length function $\ell : W \to \NN$.
Define a \emph{word} to be any finite sequence of elements of $S$.
There exists a unique associative operation $\circ : W\times W \to W$
such that  
$v\circ w = vw$ if $\ell(vw) = \ell(v) + \ell(w)$ and $s\circ s = s$ if $s \in S$ \cite[Theorem 7.1]{Humphreys}.
One way of defining a \emph{reduced word} for $w \in W$ is as a word $(s_1, s_2,\dots, s_n)$
of minimal possible length $n$ such that $w = s_1 \circ s_2 \circ \cdots \circ s_n$.
Let $\cR(w)$ be the set of reduced words for $w$. It is a fundamental result of Matsumoto (see \cite[\S1.2]{GP}) that $\cR(w)$ is spanned and preserved by 
the \emph{braid relations} of $(W,S)$, which one defines as the symmetric relations on words of the form
\be
\label{braid-eq}
 (\text{ --- }, \underbrace{s,t,s,t,s,\dots}_{m\text{ terms}},\text{ --- }) \sim (\text{ --- },\underbrace{t,s,t,s,t,\dots}_{m\text{ terms}},\text{ --- }) 
\ee
for $s, t \in S$ with $m=m(s,t)<\infty$,
where $m(s,t)$ is the order of the product $st \in W$
and
 the
corresponding `` --- '' symbols on the left and right stand for arbitrary, identical subsequences.

Let $* \in \Aut(W)$ be an involution with $S=S^*$.
We write $w^*$ for the image of $w \in W$ under $*$, and call $(W,S,*)$ a \emph{twisted Coxeter system}.
Consider the following variation of $\cR(w)$:
 define $\iR_*(w)$ as the set of words $(s_1,s_2,\dots,s_n)$
 of minimal possible length $n$ such that 
$
 w = s_n^* \circ \cdots \circ s_2^* \circ s_1^* \circ s_1 \circ s_2 \circ \cdots \circ s_n.
 $
This set is nonempty if and only if $w$ belongs to the set of \emph{twisted involutions}
$\I_* = \I_*(W) = \{ w \in W : w^* = w^{-1}\}.$
The sequences in $\iR_*(w)$ are the appropriate ``involution'' analogue of reduced words,
and have been studied in a few different places under various names and left/right conventions.
See, for example, the papers of Richardson and Springer \cite{RichSpring,RichSpring2},
Hultman \cite{H1,H2,H3}, Hu and Zhang \cite{HuZhang1,HuZhang2}, and Hamaker, Marberg, and Pawlowski \cite{HMP1,HMP2}. Following \cite{HMP1}, we refer to 
the elements of $\iR_*(w)$ as \emph{involution words}.
These objects naturally come up in the study of certain symmetric varieties \cite{CJ,CJW,WY} and 
 Iwahori-Hecke algebra modules \cite{L0,L2,L3}.

Some recent papers \cite{HMP2, HuZhang1, HuZhang2,HuZhang3}
have considered the problem of finding a set of ``involution braid relations'' which span and preserve  $\iR_*(w)$ for all $w \in \I_*$.
The associativity of the product $\circ$ implies that
each of these sets is preserved 
by the ordinary braid relations \eqref{braid-eq},
but these usually fail to span $\iR_*(w)$.
For example, suppose $s,t \in S$ and $m\geq 1$ are such that the $m$-element words $(s,t,s,t,\dots)$ and $(t,s,t,s,\dots)$ both belong to $\iR_*(z)$
for some $z \in \I_*$. As explained in Section~\ref{rel-sect},  this occurs if and only if $m=m_*(s,t)$ as defined by \eqref{theta-eq},
in which case $m$ is  
either $\frac{1}{2} m(s,t)$ or $\frac{1}{2} m(s,t)+\frac{1}{2}$ or $\frac{1}{2} m(s,t)+1$ or $m(s,t)$. 
%
Under this condition, the symmetric word relation 
\be
\label{folded-eq}
( \underbrace{s,t,s,t,\dots}_{m\text{ terms}},\text{ --- }) \sim (\underbrace{t,s,t,s,\dots}_{m\text{ terms}},\text{ --- }) 
\ee
preserves $\iR_*(w)$, but not necessarily $\cR(w)$. We refer to the relations of this type with $m<m(s,t)$
as the \emph{half-braid relations} of $(W,S,*)$.

\begin{example} \label{ex-1}
Suppose $W$ is the symmetric group $S_4$ and $S = \{s_1,s_2,s_3\}$ where $s_i = (i,i+1)$, so that $(W,S)$ has type $A_3$.
 Let $*=\id$. 
The  half-braid relations of $(W,S,*)$ are 
$ (s_1,s_2,\text{ --- }) \sim (s_2,s_1,\text{ --- })$ and $(s_2,s_3,\text{ --- }) \sim (s_3,s_2,\text{ --- }).$
These and the braid relations of $(W,S)$ span $\iR_*(w)$ for all $w \in \I_*$.
The elements of $\iR_*(4321)$ are 
$ (1,3,2,1) \sim (\textbf3,\textbf1,2,1)\sim (3,\textbf2,\textbf1,\textbf2) \sim(\textbf2,\textbf3,1,2) 
\sim (2,\textbf1,\textbf3,2) \sim (\textbf1,\textbf2,3,2)\sim (1,\textbf3,\textbf2,\textbf3)
\sim
(\textbf3,\textbf1,2,3),
$
where  we write $(i,j,\dots)$ in place of $(s_i,s_j,\dots)$.
 \end{example}

We define $(W,S,*)$ to be \emph{perfectly braided} if, as in the preceding example,
each set $\iR_*(w)$ for $w \in \I_*$ is an equivalence class under the transitive relation generated by the ordinary and half-braid relations \eqref{braid-eq} and \eqref{folded-eq}.
The following, surprising fact is equivalent to \cite[Theorem 3.1]{HuZhang1}.

 \begin{theorem}[Hu and Zhang \cite{HuZhang1}] \label{hz-thm}
 If $(W,S)$ has type $A_n$ and $*=\id$, then the 
 twisted Coxeter system $(W,S,*)$ is perfectly braided.
  \end{theorem}


A more general, but weaker result is shown in \cite{HMP2}.
Let $s, t\in S$ be distinct elements and suppose $r_1,\dots, r_k \in S$ and $m \in \PP$ are such that $(k+m)$-element words
$(r_1,\dots,r_k,s,t,s,\dots)$ and $(r_1,\dots,r_k, t,s,t,\dots)$ belong to $\iR_*(z)$
for a common element $z \in \I_*$. The symmetric word relation  
\be
\label{suff-eq}
 (r_1,r_2,\dots,r_k, \underbrace{s,t,s,\dots}_{m\text{ terms}},\text{ --- }) \sim (r_1,r_2,\dots,r_k,\underbrace{t,s,t,\dots}_{m\text{ terms}},\text{ --- }).
\ee
then preserves $\iR_*(w)$ for all $w \in \I_*$. We refer to the relations arising in this way as the \emph{generalized half-braid relations} for $(W,S,*)$.
The following is equivalent to \cite[Theorem 7.9]{HMP2}.

\begin{theorem}[Hamaker, Marberg, and Pawlowski \cite{HMP2}]
\label{suff-thm}
In any twisted Coxeter system $(W,S,*)$, each
set $\iR_*(w)$ for $w \in \I_*$ is spanned and preserved by the generalized half-braid relations \eqref{suff-eq}.
\end{theorem}

The relations prescribed by this theorem are usually highly redundant and  not easily determined from the Coxeter diagram of $(W,S)$.  
Our main  result gives a substitute for Theorem~\ref{suff-thm} without these defects for the cases when $W$ is a finite or affine Coxeter group.
Let $\Gamma$ be the Coxeter diagram of $(W,S)$.
Define an \emph{induced copy} of ${^2A_3}$ in 
$(W,S,*)$ to be an induced subgraph of $\Gamma$ of the form
\[
\barr{c} \athreediagram \\ A_3 \earr
\]
with $a^* =b $ and $b^* = a$ and $x^*=x$. 
Similarly, let an \emph{induced copy} of $B_3$ or $H_3$ or $D_4$ in $(W,S,*)$
refer to an induced subgraph of $\Gamma$
 of one of the respective forms 
\[
\barr{c} \bthreediagram \\ B_3 \earr
\qquord
\barr{c} \hthreediagram \\ H_3 \earr
\qquord
\barr{c} \dfourdiagram \\ D_4 \earr
\]
whose vertices are fixed pointwise by $*$. 
We define the \emph{exceptional braid relation} of each of these induced subgraphs to be the symmetric word relation given 
(somewhat arbitrarily) by:
\begin{itemize}
\item $(x,a,b,x,\text{ --- } ) \sim (x,a,x,b,\text{ --- }) $ in type $^2A_3$.
\item $(x,a,b,x,a,b,\text{ --- } ) \sim (a,b,x,a,b,x,\text{ --- } )$ in type $B_3$.
\item $(x,a,b,x,a,b,x,a,b,\text{ --- } )\sim (a,b,x,a,b,x,a,b,x,\text{ --- })$ in type $H_3$.
\item $(x,a,b,c,x,a,b,c,\text{ --- } )\sim (a,b,c,x,a,b,c,x,\text{ --- })$ in type $D_4$.
\end{itemize}
Let $\sB = \sB(W,S)$ be the set of braid relations for $(W,S)$,
let $\isB=\isB(W,S,*)$ be the union of $\sB$ with the set of half-braid relations for $(W,S,*)$,
and let $\isB^+=\isB^+(W,S,*)$ be the set of exceptional braid relations associated to
every induced copy of $^2A_3$, $B_3$, $H_3$, and $D_4$ in $(W,S,*)$.
The following is our main theorem.

\begin{theorem}\label{main-thm}
Let $(W,S,*)$ be a twisted Coxeter system and suppose
each irreducible component of $(W,S)$ is of finite or affine type. Then every set  $\iR_*(w)$ for $w \in \I_*$
is an equivalence class under the transitive relation generated by $ \isB \cup \isB^+$.
\end{theorem}

When $*=\id$ and $(W,S)$ has type $A_n$, $B_n$, $D_n$, or $F_4$, this statement
is equivalent to the main results of
 Hu, Wu, and Zhang in \cite{HuZhang1,HuZhang2,HuZhang3}.
Our proof  of the theorem is computational, and derives from a general algorithm
for reducing the set of generalized half-braid relations for $(W,S,*)$ to a 
finite superset of the braid relations $\sB$. This algorithm is described in Section~\ref{alg-sect}.

\begin{remark}
The set $\isB \cup \isB^+$ technically includes redundant relations when $(W,S,*)$ contains induced copies of 
$^2A_3$ or $D_4$, but one can
easily exclude these to obtain a minimal spanning set. 
\end{remark}

\begin{example}
Suppose $(W,S)$ is the affine Coxeter system of type $\tilde C_n$,
with $S = \{s_0,s_1,\dots,s_n\}$.
If $*$ is the involution with $s_i \leftrightarrow s_{n-i}$ for all $i$, then
$\isB^+ = \varnothing$ for $n\in \{2,3\}$ and all odd $n\geq 5$. For $n=4$ the Coxeter diagram 
$0 \overset{4}{=\hspace{-2mm}=} 1\text{ --- }2 \text{ --- }3 \overset{4}{=\hspace{-2mm}=} 4$ 
contains two induced copies of $^2A_3$
and $\isB^+$
consists of the relations $(2,1,3,2,\text{ --- })\sim (2,1,2,3,\text{ --- })$
and $(2,3,1,2,\text{ --- })\sim (2,3,2,1,\text{ --- })$.
\end{example}

A twisted Coxeter system $(W,S,*)$ is \emph{irreducible} if $*$ acts transitively on the connected 
components of the Coxeter diagram of $(W,S)$.
As a corollary, we get this generalization of Theorem~\ref{hz-thm}:

 \begin{corollary}\label{main-cor}
Suppose
$(W,S,*)$ is an irreducible twisted Coxeter system of affine or finite type. 
Then $(W,S,*)$ is perfectly braided if and only if 
either (i) $s^* \neq s$ for all $s \in S$,
(ii) $*=\id$ and $(W,S)$ has type $A_n$ or $\tilde A_n$,
or 
(iii) $(W,S)$ has type $\tilde A_2$, $\tilde C_2$, $\tilde G_2$, or $I_2(m)$ for $3 \leq m \leq \infty$.
 \end{corollary}
 
 \begin{proof}
Combine Theorem~\ref{main-thm} with the classification of Coxeter graphs of positive type in \cite{Humphreys}.
 \end{proof}
 
 Theorem~\ref{main-thm} and Corollary~\ref{main-cor} are nice enough to inspire some general conjectures.

 \begin{conjecture}
A twisted Coxeter system $(W,S,*)$ is perfectly braided if $s^*\neq s $ for all $s \in S$.
\end{conjecture}

 \begin{conjecture}
There exists a finite subset $\sC$ of the generalized half-braid relations for $(W,S,*)$
 such that the relations $\isB \cup \sC$
span and preserve  $\iR_*(w)$  for each $w \in \I_*$.
  \end{conjecture}
  
Call a subset $J \subset S$ a \emph{neighborhood} if 
$J = \{ t \in S : m(s,t) \neq 2\}$
for some $s \in S$.

\begin{conjecture}
Let $\sC$ be a subset of the generalized half-braid relations for $(W,S,*)$.
Suppose the relations $\isB \cup \sC$ span $\iR_*(w)$ for each neighborhood $J\subset S$ with $J=J^*$ and each
$w \in \I_* \cap W_J$.
Then the same relations span $\iR_*(w)$ for all $w \in \I_*$.
\end{conjecture}


We note one other result.
For $w \in \I_*$ define 
$ \iHR_*(w)$ as the union $\bigcup_{v} \cR(v)$ over $v \in W$ with $(v^{-1})^*\circ v =w$.
We refer to elements of $\iHR_*(w)$ as \emph{involution Hecke words}.
Such words were studied in \cite{HMP2};
in type $A_n$, they are closed related to
the so-called \emph{Chinese monoid}.
Clearly $\iHR_*(w)$ contains $\iR_*(w)$ and
 is preserved by the usual braid relations.
 When $s,t \in S$ and $n \in \PP$ are such that $m_*(s,t) \leq n <  m(s,t) < \infty$,
 define $\sim$ and $\approx$ as the symmetric word relations with
\be\label{mixed-eq}
 (\underbrace{t,s,t,\dots}_{n\text{ terms}},r_1,r_2,\dots,r_k)
 \sim (\underbrace{s,t,s,\dots}_{n\text{ terms}},r_1,r_2,\dots,r_k)
\approx (\underbrace{s,t,s,t,\dots}_{n+1\text{ terms}},r_1,r_2,\dots,r_k)
\ee
for each 
$(r_1,r_2,\dots,r_k) \in \cR(v)$ and $v \in W$ with  $\ell(sv) = \ell(tv) > \ell(v)$.
Let $\ihB = \ihB(W,S,*)$ be the union of the set of 
relations $\sim$ and $\approx$  arising in this way
 and the usual set of braid relations $\sB$ for $(W,S)$.
It follows from Theorem~\ref{suff-thm} that the relations $\ihB$ preserve the sets $\iHR_*(w)$.
The following is a consequence of Proposition~\ref{perfect-prop2}.
In type $A_n$, this statement is equivalent to \cite[Theorem 6.4]{HMP2}.

\begin{proposition}\label{perfect-prop}
If   $(W,S,*)$ is perfectly braided, then  $\ihB$ spans $\iHR_*(w)$ for each $w \in \I_*$.
\end{proposition}

%

In Section~\ref{prelim-sect} we review a few standard facts about Coxeter systems.
Section~\ref{rel-sect} describes some more detailed properties of generalized half-braid relations. 
Section~\ref{alg-sect} defines our main algorithms,
and Sections~\ref{red-sect} and \ref{app-sect} 
 indicate how we use their output to verify Theorem~\ref{main-thm}.

\subsection*{Acknowledgements}

I thank
Zach Hamaker, Luca Moci, Brendan Pawlowski, and Yan Zhang for helpful conversations.

\section{Preliminaries}\label{prelim-sect}

We write $\ZZ$, $\NN$, and $\PP$ for the sets of all, nonnegative, and positive integers, and define $[n] = \{1,2,\dots,n\}$ for $n \in \NN$.
Let $(W,S,*)$ be a twisted Coxeter system. 
Our main reference for the following material is \cite[Chapter 5]{Humphreys}.

Define $V$ as the vector space over $\RR$ with a basis given by
 $\alpha_s$ for $s \in S$.
 Write $(\cdot,\cdot)$ for the symmetric, bilinear form on $V$ with 
$(\alpha_s,\alpha_t) = -\cos\(\tfrac{\pi}{m(s,t)}\)$ for $s,t \in S$,
 where $m(s,t)$ is the order of $st \in W$.
 The formula $s  v = v  -2(\alpha_s,v)\alpha_s$ for $s \in S$ and
 $v\in V$ extends to a faithful action of $W$ on $V$ which preserves $(\cdot,\cdot)$.
We refer to $V$ with this $W$-module structure as the \emph{geometric representation} of $(W,S)$.
The \emph{root system} of $(W,S)$ is the set
$\Phi = \{ w\alpha_s : w \in W,\ s\in S\}.$
 This set is the disjoint union $\Phi = \Phi^+ \sqcup \Phi^-$
where $\Phi^+ = \Phi \cap \RR^+ \spanning\{ \alpha_s : s \in S\}$
and $\Phi^- = - \Phi^+$.
We write $<$ for the Bruhat order on $W$.
Recall that if $w \in W$ and $s \in S$ then the following are equivalent:
(1) $ws < w$, (2) $sw^{-1} < w^{-1}$, (3) $\ell(ws) = \ell(w)-1$, and (4) $w\alpha_s \in \Phi^-$.
Let $\DesR(w) = \{s \in S : ws<w\}$ and $\DesL(w) = \{s \in S : sw<w\}$.

\begin{lemma}\label{commute-lem}
If $w \in W$ and $s,t \in S$ then $ws=tw$ if and only if $w\alpha_s \in \left\{ \alpha_{t},-\alpha_{t}\right\}$.
\end{lemma}

\begin{proof}
This is a straightforward consequence of \cite[Proposition 5.6 and Lemma 5.7]{Humphreys}.
\end{proof}

\begin{corollary}\label{commute-cor}
If $w \in W$ and $s \in S$ then $ws=s^*w < w$ if and only if $w\alpha_s = -\alpha_{s^*}$.
\end{corollary}


Recall that $\I_* =\I_*(W)= \{ w \in W : w ^{-1} = w^*\}$.

\begin{lemma}\label{x-lem}
If $w \in \I_*$ and $s \in S$ then  $\ell(s^*ws)  =\ell(w)$ if and only if $s^*ws=w$.
\end{lemma}

\begin{proof}
This follows from the exchange condition; cf. \cite[Lemma 3.4]{H2}.
\end{proof}

\begin{corollary}\label{notcommute-lem}
If $w \in \I_*$ and $s \in S$ then $s^*ws < w$ if and only if $w\alpha_s \in \Phi^- \setminus\{-\alpha_{s^*}\}$.
\end{corollary}

\begin{proof}
Since $s^*ws < w$ if and only if $s^* w \neq ws < w$, the result follows from Corollary~\ref{commute-cor}.
\end{proof}

Suppose $w \in \I_*$ and $s \in \DesR(w)$.
Corollaries~\ref{commute-cor} and \ref{notcommute-lem}
 imply that a unique element $v \in \I_*$
 exists with $v < s^* \circ v \circ s = w$. 
It follows by induction that the set $\iR_*(w)$ is nonempty, and moreover
that $w \in \I_*$ has an involution word ending in $s$ whenever $s \in \DesR(w)$.
We write
$\ellhat_* : \I_* \to \NN$
for the function which assigns to $w \in \I_*$ the common length of each element of $\iR_*(w)$.

\section{Relations}\label{rel-sect}

Let $(W,S,*)$ be a twisted Coxeter system, and
recall the definition of the generalized half-braid relations \eqref{suff-eq}.
Fix  elements $s,t \in S$ and suppose $\theta$ is a map $\{s,t\}\to W$.
Define
\be\label{theta-eq}
m_\theta(s,t) = \begin{cases} 
\frac{1}{2} m(s,t)+\frac{1}{2}&\text{if $m(s,t)$ is odd and $\theta(\{s,t\}) = \{s,t\}$} \\
\frac{1}{2}m(s,t) + 1&\text{if $m(s,t)$ is even and $\theta(s)=s$ and $\theta(t)=t$} \\
\frac{1}{2}m(s,t)&\text{if $m(s,t)$ is even and $\theta(s)=t$ and $\theta(t)=s$}\\
m(s,t)&\text{otherwise}.
\end{cases}
\ee
If  $m(s,t) = \infty$ then $m_\theta(s,t) = \infty$.
If $m_\theta(s,t)<m(s,t)$, then $\theta$ extends to an involution of  $W_{\{s,t\}} = \langle s,t\rangle$
and $m_\theta(s,t) = \ellhat_\theta(\Delta)$ for the longest element $\Delta \in \I_\theta(W_{\{s,t\}})$ \cite[Proposition 7.7]{HMP2}.

\begin{proposition}\label{hb-prop}
There exists $z \in \I_*$ such that $\iR_*(z)$ contains both of the words 
\be\label{words}  (r_1,r_2,\dots,r_k, \underbrace{\dots,s,t,s,t,s}_{m\text{ terms}}) \qquand (r_1,r_2,\dots,r_k,\underbrace{\dots,t,s,t,s,t}_{m\text{ terms}})
\ee
if and only if there exists $y \in \I_*$ such that $(r_1,r_2,\dots,r_k) \in \iR_*(y)$,  $\{s,t\} \subset S - \DesR(y)$, and $m = m_\theta(s,t) < \infty$
where $ \theta : W \to W$ denotes the map $w\mapsto (ywy^{-1})^*$.
\end{proposition}


\begin{proof}
Let $w \in W$. 
When $m(s,t) <\infty$, define $\Delta$ as the longest element of $W_{\{s,t\}} $.
If $\{s,t\} \subset \DesR(w)$, then $m(s,t)<\infty$
and $\ell(w\Delta) = \ell(w) - \ell(\Delta)$  \cite[Lemma 1.2.1]{GP},
while if $\{s,t\}\cap \DesR(w)=\varnothing$ and $m(s,t) <\infty$, then $\ell(w\Delta) = \ell(w)  + \ell(\Delta)$.
Similar left-handed properties hold.

Suppose $m(s,t) <\infty$ and $y \in \I_*$ is such that 
 $\{s,t\} \cap \DesR(y) = \varnothing$, so that
$y' = \Delta^* y$  has length $\ell(y') =\ell(y) + \ell(\Delta)$.
The exchange condition
implies that if
$s \in \DesR(y')$ then $ys \in \{s^*y,t^*y\}$
while if $t \in \DesR(y')$ then $yt \in \{s^*y,t^*y\}$.
Writing $s'=(ysy^{-1})^*$ and $t'=(ysy^{-1})^*$,
we deduce that there exists a unique element $z \in \I_*$ such that one of the following occurs:
\ben
\item[(1)] $z = \Delta^* y \Delta$, $\ell(z) = \ell(y)+2\ell(\Delta)$, and $\{s',t'\}\cap \{s, t\} = \varnothing$.
\item[(2)] $z = y \Delta = \Delta^* y$, $\ell(z) = \ell(y) + \ell(\Delta)$, and $\{s',t'\}=\{s, t\}$.
\item[(3)] $z \in \{ \Delta^* ys\Delta, \Delta^* yt\Delta \} \cap \{ \Delta^* s^* y\Delta, \Delta^* t^* y\Delta\}$,
$\ell(z) = \ell(y) + 2\ell(\Delta)-1$, 
 and $|\{s',t'\}\cap \{s, t\}| = 1$.
\een
Define $\theta(w)= (ywy^{-1})^*$.
In each case, one checks that
 if $(r_1,r_2,\dots,r_k)\in \iR_*(y)$ then both words in \eqref{words} belong to $\iR_*(z)$
 when $m=m_\theta(s,t)$.

Assume conversely that both words in \eqref{words} belong to $\iR_*(z)$ for $z \in \I_*$.
Then $(r_1,r_2,\dots,r_k) \in \iR_*(y)$ for some $y \in \I_*$. We must have $\{s,t\} \cap \DesR(y) = \varnothing$
and $m(s,t) <\infty$ since $m>0$ and  $\{s,t\} \subset \DesR(z)$.
Again define $\theta(w) = (ywy^{-1})^*$.
The argument above shows that 
the $(k+m_\theta(s,t))$-element words $(r_1,r_2,\dots,r_k,\dots,s,t,s)$ and $(r_1,r_2,\dots,r_k,\dots,t,s,t)$ both belong to 
$\iR_*(z')$ for some $z' \in \I_*$.
Since $\{s,t\}\subset \DesR(z)\cap \DesR(z')$, we must have $z=z'$ and $m=m_\theta(s,t)$. 
\end{proof}

In the course of the preceding proof, we established the following:

\begin{corollary}\label{twodescent-cor}
Let $s,t \in S$ and $z \in \I_*$. If $\{s,t\}\subset \DesR(z)$,
then there is a unique positive integer $m\leq m(s,t)<\infty$ such that 
the words in \eqref{words}
both belong to $\iR_*(z)$ 
for some $r_1,r_2,\dots,r_k \in S$.
\end{corollary}

Recall the definition of $\iHR_*(w)$ and $\isB$ and $\ihB$
from Theorem~\ref{main-thm} and Proposition~\ref{perfect-prop}.

\begin{proposition}\label{perfect-prop2}
Suppose $(W,S,*)$ is perfectly braided.
Let $z \in \I_*$ and $\s \in \iHR_*(z)$.
There exists a sequence
 $\sim_1$, $\sim_2$, \dots, $\sim_l$
of
 relations in $\ihB$
and
a sequence 
$\s_0$, $\s_1$, \dots, $\s_l \in \iHR_*(z)$
of
 reduced words with weakly decreasing lengths
such that $\s = \s_0 \sim_1 \s_1 \sim_2 \cdots \sim_l \s_l \in \iR_*(z)$.
\end{proposition}

In the following proof, we write $\a\b$ for the concatenation of two words $\a$ and $\b$.

\begin{proof}
Let $\s = (s_1,s_2,\dots,s_k) \in \iHR_*(z)$.
Assume $k>0$
and $(s_1,s_2,\dots,s_{k-1}) \in \iHR_*(y)$ for $y \in \I_*$.
By induction there are relations $\sim_1$, $\sim_2$, \dots, $\sim_l$ in $\ihB$ and words
$\a_0$, $\a_1$, \dots, $\a_l \in \iHR_*(y)$ with weakly decreasing lengths
such that $(s_1,s_2,\dots,s_{k-1}) = \a_0 \sim_1 \a_1 \sim_2 \cdots \sim_l \a_l \in \iR_*(y)$.
Let $\b = (s_k)$.
Note that 
 if $\a_{i-1} \b$ and $\a_{i}\b$ are both reduced words, then we also have 
$\a_{i-1}\b \sim_i \a_{i}\b$. 

If $y<s_k^* \circ y \circ s_k = z$ then 
each $\a_i\b$ is reduced and $\a_l \b \in \iR_*(z)$, so 
the result holds with $\s_i = \a_i \b$.
Suppose $y = z$ so that $s_k \in \DesR(z)$.
Since   the relations $\ihB$ span $\iR_*(z)$, 
we may assume that $\a_l \in \iR_*(z)$ has $s_k$ as its last entry.
Let $i \in [l]$ be minimal
such that $\a_{i-1}\b$ is reduced but  $\a_i \b$ is not reduced.
The relation $\sim_i$ cannot belong to $\sB$,
so
 $\a_{i-1}$ must be an $(n+j)$-element word
 $
\a_{i-1} = (s,t,s,\dots,r_1,\dots,r_j)
$,
where $s,t \in S$ and  $n \leq m(s,t)< \infty$ and  $(r_1,\dots,r_j) \in \cR(v)$ for
some $v \in W$ with $\ell(sv) = \ell(tv) > \ell(v)$,
and $\a_i$ must be the $(n+j)$ element word
$
 (t,s,t,\dots,r_1,r_2,\dots,r_j)
$ 
or the $(n+j-1)$-element word
$(s,t,s,\dots,r_1,\dots,r_j)$.
Let $\Delta$ be the longest element in $W_{\{s,t\}}$.
Then $\ell(\Delta v) = \ell(\Delta) + \ell(v)$ and $s_k $ belongs to $ \DesR(\Delta v)$ but not $\DesR(v)$,
so the exchange principle implies that $\Delta v s_k \in \{ \Delta s v, \Delta tv \} = \{ s\Delta v, t\Delta v\}$
and $vs_k \in \{sv, tv\}$.
It follows that  a sequence of braid relations transforms  $\a_{i-1}\b$ to the $(n+j+1)$-element word
$
(s,t,s,t,\dots,r_1,r_2,\dots,r_j).
$
A relation 
 in $ \ihB$ transforms this word to $\a_{i-1}$,
so 
the desired property follows.
\end{proof}

\section{Algorithms}
\label{alg-sect}

The goal of this section is to describe an algorithm which can be used to verify Theorem~\ref{main-thm}
by a computer calculation.
Choose a twisted Coxeter system $(W,S,*)$ and write  $\I_* = \I_*(W)$.
We assume  $S$ is finite, and fix an arbitrary total ordering of its elements.
  Everything in this section will be defined relative to these choices, though this dependence is often suppressed in our notation.
A \emph{word} is a finite sequence $(s_1,s_2,\dots,s_n)$ with $s_i \in S$. An \emph{involution word} is a word in $\iR_*(w)$ for some $w\in \I_*$.
As in the introduction, we often write
$
(s_1,s_2,\dots,s_n,\text{ --- }) \sim (t_1,t_2,\dots,t_n,\text{ --- })
$
as a shorthand for the symmetric relation on words with  $\a \sim \b$
if   there exist any number of elements $c_i \in S$ such that
$\a=(s_1,s_2,\dots,s_n,c_1,c_2,\dots,c_m)$ and $\b=(t_1,t_2,\dots,t_n,c_1,c_2,\dots,c_m)$.

For each $p,q \in \PP$ with $q\geq 2$, let $\sA_{p,q}= \sA_{p,q}(W,S,*)$ 
be the subset of generalized half-braid relations \eqref{suff-eq}
for $(W,S,*)$ with $p=k+m$ and $q= m(s,t)$.
Define 
\be\label{gen-rel-eq}\sA=\sA(W,S,*) = \bigcup_{p \in \PP}\bigcup_{2\leq q<\infty} \sA_{p,q}\ee
and also let
$\sAq{q} = \bigcup_{i \in \PP} \bigcup_{q \leq j <\infty} \sA_{i,j}$ and $ \sAp{p} = \bigcup_{i \in [p]}\bigcup_{2\leq j <\infty} \sA_{i,j}.$
By Theorem~\ref{suff-thm}, the relations in $\sA$ span and preserve 
 $\iR_*(w)$ for all $w \in \I_*$.
Suppose $\sim$ is a word relation and $\sX \subset \sA \cup \sB$. We write $\sX \Rightarrow {\sim}$ if there are relations $\sim_1$, $\sim_2$, \dots, $\sim_n$ in $\sX$ 
such that whenever two involution words $\a$, $\b$ satisfy  $\a \sim \b$, there are
involution words $\a_0$, $\a_1$, \dots, $\a_n$ with $\a=\a_0 \sim_1 \a_1 \sim_2\dots \sim_n \a_n = \b$.
If $\sY \subset \sA \cup \sB$ then we write $\sX \Rightarrow \sY$ when $\sX \Rightarrow {\sim}$ for all ${\sim} \in \sY$.

\begin{lemma}\label{relation-implication-lem}
Let $\r=(r_1,\dots,r_k)$, $\s=(s_1,\dots,s_n)$, and $\t=(t_1,\dots,t_n)$ be words with $n>0$. 
Set $p=k+n$ and $q=m(s_n,t_n)$. Define $\cC$ and $\cD$ as the equivalence classes of $\r\s$ and $\r\t$
under the transitive relation generated by $\sAp{p-1} \cup \sB$. 
Suppose both
(a) some $z \in \I_*$ exists with $\{ \r\s,\r\t\} \subset \iR_*(z)$,
and 
(b) some $\u \in \cC$ and $\v \in \cD$
exist with right-most entries $s $ and $t$ such that $s=t$ or
$q<m(s,t)<\infty$.
It then holds that $\sAp{p-1}\cup \sAq{q+1}\cup \sB \Rightarrow {\sim}$
for the word relation ${\sim}$ defined by
$(r_1,\dots,r_k,s_1,\dots,s_n,\text{ --- }) \sim (r_1,\dots,r_k,t_1,\dots,t_n,\text{ --- })$.
\end{lemma}

\begin{proof}
Let $\approx$, $\sim_1$, and $\sim_2$ be the transitive relations
respectively generated by $\sAp{p-1}\cup \sAq{q+1}\cup \sB$,
  $\sAp{p-1}\cup \sB$,
  and $ \sAq{q+1}\cup \sB$.
By Theorem~\ref{suff-thm},  
 it  suffices to show that $\u \approx \v$.
Since $\{s,t\}\subset \DesR(z)$,
it follows by Corollary~\ref{twodescent-cor} that there exist $\x, \y \in \iR_*(z)$ with either $\x=\y$ (if $s=t$) 
or $\x \sim_2 \y$ (if $q<m(s,t)$)
and such that $\x$ ends in $s$ and $\y$ ends in $t$.
Let $x',y' \in \I_*$ be the elements with $x' < s^* \circ x' \circ s= z$ and $y' <  t^* \circ y' \circ t = z$,
and let $\u',\v',\x',\y'$ be the words formed by omitting the last entries in $\u,\v,\x,\y$ respectively.
By construction,  $\{\u',\x'\} \subset \iR_*(x')$ and $\{\v',\y'\} \subset \iR_*(y')$,
so $\u' \sim_1 \x'$ and $\v' \sim_1 \y'$
by Theorem~\ref{suff-thm}. 
It follows that $\u\sim_1 \x  $ and $ \y \sim_1 \v$, so  $\u \approx \v$.
\end{proof}

\begin{lemma}\label{commute-impl-lem}
Let $r_1,\dots,r_k,s,t \in S$. Assume $k>0$ and $m(s,t) = 2$
and suppose the words $(r_1,\dots,r_k,s)$, and $(r_1,\dots,r_k,t)$ both belong to $\iR_*(z)$ for some $z \in \I_*$.
Then $\sAp{k}\cup \sAq{3}\cup \sB \Rightarrow {\sim}$ 
for the word relation ${\sim}$ defined by
$(r_1,\dots,r_k,s,\text{ --- }) \sim (r_1,\dots,r_k,t,\text{ --- }).$
\end{lemma}

\begin{proof}
Let $y \in \I_*$  be 
such that $(r_1,\dots,r_k) \in \iR_*(y)$. Since $k>0$ we have $y\neq 1$.
It follows from Proposition~\ref{hb-prop} 
that $\{s,t\} \cap \DesR(y) = \varnothing$ and 
$ysy^{-1} = t^*$ and $yty^{-1} = s^*$.
Lemma~\ref{commute-lem} implies that 
$y \alpha_s = \alpha_{t^*}$ and $y \alpha_t = \alpha_{s^*}$, and hence $z = s^*\circ y \circ s = s^*t^*y$.
Let $r=r_k \in \DesR(y)$.
One checks that 
$z \alpha_r = s^*t^*y\alpha_r\in \Phi^-$, so $\{r,s,t\}\subset \DesR(z)$.
Therefore both $m(r,s)$ and $m(r,t)$ are finite.

Define $\approx$, $\sim_1$, and $\sim_2$ as the transitive relations respectively generated by $\sAp{k}\cup \sAq{3}\cup \sB$,
$\sAp{k}\cup \sB$,
and
$ \sAq{3}\cup \sB$.
Let $\u = (r_1,\dots,r_{k-1},r,s)$ and $\v = (r_1,\dots,r_{k-1},r,t)$ so that $\u,\v \in \iR_*(z)$. It suffices to check that $\u \approx \v$. For this, it is enough to find $\a,\b\in \iR_*(z)$ which both end in $r$ and satisfy $\u\approx \a$ and $\v\approx \b$, since
then   $\a \sim_1 \b$ by Theorem~\ref{suff-thm}.
If $m(r,s)=2$ then we can take $\a=(r_1,\dots,r_{k-1},s,r)$. If $m(r,s) \geq 3$ then 
it follows from Corollary~\ref{twodescent-cor} that $\iR_*(z)$ contains two  words of the form
$\u' = (r_1',\dots,r_j',\underbrace{\dots,s,r,s,r,s}_{m\text{ terms}}) $ and $ \a = (r_1',\dots,r_j',\underbrace{\dots,r,s,r,s,r}_{m\text{ terms}})$
where $j \leq k-1$ and $m\geq 2$. In this case, we have $\u \sim_1 \u'$ by Theorem~\ref{suff-thm} and $\u'\sim_2 \a$ by definition, so $\u\approx \a$ as desired.  
The word $\b \in \iR_*(z)$ with $\v\approx \b$ is constructed by a similar argument.
\end{proof}

Given $J\subset S$, let $V_J = \RR\spanning \{ \alpha_s : s \in J\} \subset V$ and $W_J = \langle s \in J\rangle\subset W$, where $V$ is defined as  in Section~\ref{prelim-sect}.
We write $\RR[S] = \RR[x_s : s \in S]$ for the polynomial ring over $\RR$ in a commuting set of indeterminates indexed by $S$,
and frequently consider the free $\RR[S]$-module $V\otimes_\RR \RR[S]$, whose elements are just linear combinations of the simple roots $\{\alpha_s : s \in S\}$ with polynomial coefficients.
We view each $w \in W$ as a linear automorphism of this module by identifying $w$ with $w\otimes_\RR 1$, and extend
$(\cdot,\cdot)$ to an $\RR[S]$-bilinear form on $V\otimes_\RR \RR[S]$ with values in $\RR[S]$ in the obvious way.

\begin{definition}
A \emph{braid system} 
for $(W,S,*)$
consists of the following data:
\begin{itemize}
\item Two words $\s = (s_1,s_2,\dots,s_n)$ and $\t= (t_1,t_2,\dots,t_n)$ of the 
same length.
\item An $\RR$-linear map  $\sigma : V_J \to V\otimes_\RR \RR[S]$ for some subset $J \subset S$.

\item A set of constraints $C$, consisting of equalities and inequalities involving elements of $\RR[S]$.
\end{itemize}
We write $(\{\s, \t\}, \sigma, C)$ to succinctly refer to this data. 
The sequences $\s$ and $\t$ are the \emph{words} of the braid system,
and
the set $J\subset S$ is the system's \emph{domain}.
\end{definition}

Let $\cB=(\{\s, \t\}, \sigma, C)$ be a braid system with domain $J$.
The system $\cB$ is \emph{constant}
if  $\sigma V_J \subset V$ 
 and $C = \varnothing$. 
We write $\alpha < 0$ (respectively, $\alpha > 0$) to indicate that  $\alpha \in V$ is a nonzero
linear combination of simple roots with nonpositive (respectively, nonnegative) coefficients.
Say that a map $f : V_J \to V$ is \emph{root-preserving}
if 
 $f(\alpha_s) < 0$ or $f(\alpha_s)>0$ for each $s \in S$.
Every $w \in W$ acts on $V$ as a root-preserving linear transformation.
A \emph{solution} to  $\cB$ 
is a ring homomorphism $\RR[S] \to \RR$ which transforms 
all constraints in $C$ to true statements, and which is such that $(1\otimes_\RR \psi) \circ \sigma : V_J \to V$ is a 
root-preserving map.
A braid system is \emph{valid}
if its words are  both reduced and it has a solution.
The following property has a slightly more involved definition.

\begin{definition}
Let $\sigma$ be a map $V_J \to V\otimes_\RR \RR[S]$.
Suppose $r,s \in J$ are such that 
$\sigma\alpha_r = \alpha_{\theta(r)^*}$ and $\sigma\alpha_s = \alpha_{\theta(s)^*}$
for a bijections $\theta : \{r,s\} \to \{r,s\}$,
and $m_\theta(s,t) < m(r,s) < \infty$.
In this case, for each $t \in J$ we refer to 
$
 (\underbrace{r,s,r,s,r,\dots}_{m_\theta(s,t)\text{ terms}}, t, \text{ --- }) \sim  (\underbrace{s,r,s,r,s,\dots}_{m_\theta(s,t)\text{ terms}},t, \text{ --- })
$
as a $\sigma$-relation. The \emph{$\sigma$-equivalence class} of a word is its equivalence class under the transitive relation generated by
$\sB$ and the $\sigma$-relations just described.
A braid system $\cB = (\{\s,\t\}, \sigma, C)$ 
 is \emph{redundant} if
  there is a word in the $\sigma$-equivalence class of $\s$ (respectively, $\t$) with right-most entry $s' \in S$ (respectively, $t' \in S$) such that
 $s'=t'$ or $m(s,t) < m(s',t') <\infty$, where $s$ and $t$ denote the right-most entries of $\s$ and $\t$.
\end{definition}

\begin{lemma}\label{3-lem}
Let $\cB $ be a braid system 
with words $\s = (s_1,\dots,s_n)$ and $\t = (t_1,\dots,t_n)$. 
Suppose $y,z \in \I_*$ are such that $\{\r\s,\r\t\}\subset \iR_*(z)$ for some
$\r=(r_1,\dots,r_k) \in \iR_*(y)$.
 Assume $n>0$, and 
let $p= \ellhat_*(z) = k+n$
and  $q = m(s_n,t_n)$.
 If $\cB$ is redundant, then $\sAp{p-1} \cup \sAq{q+1} \cup \sB \Rightarrow {\sim}$
 for the word relation ${\sim}$ given by
 $(r_1,\dots,r_k,s_1,\dots,s_n,\text{ --- }) \sim (r_1,\dots,r_k,t_1,\dots,t_n,\text{ --- })$.
\end{lemma}

\begin{proof}
If $\cB$ is redundant,  $\x\mapsto \r\x$ maps
 the $\sigma$-equivalence classes of $\s$ and $\t$ into the equivalence classes of $\r\s$ and $\r\t$ under the relation
 generated by $\sAp{p-1} \cup \sB$, so this  follows from Lemma~\ref{relation-implication-lem}.
\end{proof}

We describe five operations on braid systems.
Fix a braid system $\cB= (\{\s, \t\}, \sigma, C)$ with domain $J\subset S$.
Define $\partial J$ as the set of $s \in J-S$ such that
$m(s,t) > 2$ for some $t \in J$.
\ben
\item[(1)] When $\cB$ is constant and $r \in \partial J$, define $\cB  \downarrow r=(\{\s, \t\}, \sigma', C')$
where $\sigma' : V_{\{r\} \cup J} \to V\otimes_\RR \RR[S]$ is the unique linear map with 
$\sigma'|_{V_J} = \sigma 
$
and
$
\sigma' \alpha_r = -\sum_{s \in S}  \alpha_s \otimes_\RR x_s
$, 
and  $C'$ 
consists of the  constraints
$  x_s \geq 0$ for $s \in S$,
 $(\sigma' \alpha_r, \sigma'\alpha_s) = (\alpha_r,\alpha_s)$ for $s \in J$,
and
 $|\det \sigma| = 1$
 if  $\{r\} \cup J=S$.

\een
The constraints $(\sigma' \alpha_r, \sigma'\alpha_s) = (\alpha_r,\alpha_s)$
for $s \in J$ in $\cB \downarrow r$ form a system of
linear equations involving only the variables $x_s$ for $s \in \partial J \cup J$.
The only constraint in the system $\cB\downarrow r$ which involves any variable $x_s$
for $s \notin \partial J \cup J$ is the inequality $x_s \geq 0$.
We intentionally exclude from $\cB \downarrow r$ the natural constraint
 $(\sigma' \alpha_r, \sigma' \alpha_r) = (\alpha_r,\alpha_r) =1$
to preserve this property. 
Below, let  $\r = (r)$.
\ben
\item[(2)] Define $r \bullet_{\mathrm{asc}} \cB = (\{\s, \t\}, \sigma, C \cup C')$ where $C' = \{ \sigma \alpha_r \geq 0\}. $

\item[(3)]
Define $r \bullet_{\mathrm{des}} \cB = (\{\r\s, \r\t\}, \sigma', C \cup C')$ where 
$\sigma' = r^*  \sigma $ and 
$C' = \{ \sigma \alpha_r = -\alpha_{r^*}\}$.

\item[(4)]
Define $r \customstar \cB = (\{\r\s, \r\t\}, \sigma', C\cup C')$ where 
$\sigma' = r^*  \sigma   r|_{V_J}$ and 
$C' = \{ -\alpha_{r^*} \neq \sigma \alpha_r \leq 0\}$.

\een
Note that if $\alpha \in V\otimes_\RR \RR[S]$, then writing $\alpha \leq 0$ or $\alpha\geq 0$ is equivalent to a system of inequalities involving the coefficients of $\alpha$ in the basis of simple roots.

\ben
\item[(5)] Given a ring homomorphism $\psi : \RR[S] \to \RR[S]$, define $\cB[\psi] = (\{\s,\t\}, \sigma', C')$ where
$\sigma'  = (1\otimes_\RR \psi) \circ \sigma$ and $C'$ is formed by applying $\psi$ to all elements of $C$ and then omitting
any trivially true constraints.
Observe that if $\psi $ is a {solution}  of $\cB$ then $\cB[\psi]$ is constant.
\een
Our main algorithm is now given in three parts.

\begin{algorithm}[Expand into constant systems]\label{alg1}
Suppose we are given as inputs 
a constant braid system $\cB_0$ with domain $K \subset S$, and an element $r \in \partial K$.
Let $J = \{r\} \cup K$. Return as output the tree of braid systems $\cT(\cB_0\downarrow r)$  whose root 
vertex is
$\cB_0\downarrow r$, in which
the children of a vertex $\cB= (\{\s,\t\}, \sigma, C)$ are computed as follows:
\ben
\item If $\cB$ is constant, invalid, or redundant, then $\cB$ has no children.
\item If $\cB$ has a finite number of solutions $\psi_1,\psi_2,\dots,\psi_n : \RR[S] \to \RR$,
  then 
 the children of $\cB$ are the constant systems $\cB[\psi_1]$, $\cB[\psi_2]$, \dots, $\cB[\psi_n]$.

\item Otherwise, let $\sigma_{st} \in \RR[S]$ be such that $\sigma \alpha_s = \sum_{t \in S} \alpha_t \otimes_\RR \sigma_{st}$ for $s \in J$.
Say that $s \in J$ is a \emph{conditional descent}  (respectively, \emph{unconditional descent}) if for some $t \in S$ 
the augmented constraints $C \cup \{ \sigma_{st} > 0\}$ (respectively, $C \cup \{\sigma_{st} \geq 0\}$)  are infeasible.
If the set of unconditional descents is nonempty, then let $s$ be the least such descent 
and define the children of $\cB$ to be $s \bullet_{\mathrm{des}} \cB$ and $s \customstar \cB$.
If the set of unconditional descents is empty,
then let $s$ be the least conditional descent (if any exist) and define
 the children of $\cB$ to $s \bullet_{\mathrm{asc}} \cB$, $s \bullet_{\mathrm{des}} \cB$, and $s \customstar \cB$.
 Otherwise
let $\cB$ be its own child, so that the tree is infinite.
\een
\end{algorithm}

\begin{remark}
Each of these steps can be implemented as
a linear programming problem in a small number of variables, which can be solved by various methods.
\end{remark}

Define the \emph{descent set} of a braid system $\cB = (\{\s,\t\}, \sigma, C)$ with domain $J$
to be $\Des(\cB) = \{ s \in J: \sigma\alpha_s  \in V\text{ and }\sigma \alpha_s <0\}.$
Note that if $\sigma = w|_{V_J}$ for some $w \in W$ then $\Des(\cB) \subset \DesR(w)$.

\begin{algorithm}[Eliminate descents]\label{alg2}
Suppose we are given as input
a constant braid system $\cB$ with domain $J \subset S$.
Execute the following loop, starting with $i=0$ and $\cB_0=\cB$:
\ben
\item[1.] Given $\cB_i= (\{\s_i,\t_i\}, \sigma_i, \varnothing)$, the algorithm terminates with no output if 
the system $\cB_i$ is invalid or redundant,
or \emph{descent-periodic} in the sense defined below.
\item[2.] If instead $\Des(\cB_i)=\varnothing$ then the algorithm terminates with output $\cB_i$.
\item[3.] Otherwise, let $r $ be the least element of $\Des(\cB_i)$. 
Set $\cB_{i+1}=r \bullet_{\mathrm{des}} \cB_i$ if $\sigma_i \alpha_r = -\alpha_{r^*}$,
and otherwise
set $\cB_{i+1} =  r \customstar \cB_i$. The constraints added to $\cB_{i+1}$
are trivially satisfied, so we consider $\cB_{i+1}$ to again be a constant braid system.
Replace $i$ by $i+1$ and return to step 1.
\een
This loop may fail to terminate. When it does terminate there still may be no output returned. 
It remains to clarify step 1. 
There, we define  $\cB_i$ to be \emph{descent-periodic} 
if there exists a quasi-polynomial formula for $\sigma_i$
which shows that the sequence of minimal descents in $\Des(\cB_j)$ for $j\geq i$ is infinite and periodic.
More precisely, we consider $\cB_i$ to be descent-periodic
if for some $p,q \in \PP$ with $q\geq 2$ and $pq \leq i$ and some $r_1,r_2,\dots,r_p \in S$, the following procedure returns True:
\ben
\item[i.] If  $\s_i$ and $\t_i$ do not both begin with the word $(r_p,\dots, r_2, r_1)$ repeated $q$ times,
return False.
Otherwise, let $f_{st}^{j} \in \RR[x]$ for $(s,t,j) \in J \times S \times [p]$ 
be the unique polynomials of degree less than $q$ 
such that if $\lambda_j(n) : V_J \to V$ is the linear transformation with
$ \alpha_s \mapsto  \sum_{t \in S} f_{st}^{j}(n) \alpha_{t}$ 
then 
$ \sigma_{(i-pq)+j+np} = \lambda_j(n) $
for $j=1,2,\dots,p$ and $n=0,1,\dots,q-1$.

\item[ii.] For each $j \in [p]$ define $\beta_j = \alpha_{r_{j+1}} $ and $\cN_j = \{ \lambda_j(n) \beta_j: n \in \NN\}$,
where we set $r_{p+1} = r_1$.
If for some $j \in [p]$
there exists $\alpha \in \cN_j$ with $\alpha \not < 0$, or if $-\beta_j^* \in \cN_j \neq \{-\beta_j^*\}$, again return False.

\item[iii.] For a linear map $\lambda : V_J \to V$ and $r \in J$, define $\lambda \bullet r$ to be $r^*\lambda$ if $\lambda \alpha_r = -\alpha_{r^*}$ and $r^* \lambda r|_{V_J}$ otherwise.
Return True if for all $n \in \NN$  it holds that 
$\lambda_1(n+1) = \lambda_p(n)\bullet r_1$
and 
$\lambda_j(n+1) = \lambda_{j-1}(n+1)\bullet r_j$ for $j=2,3,\dots,p$.
 Otherwise, return False.
\een
If $\cB_i$ is descent-periodic then we cannot have $\sigma_i = w|_{V_J}$ for any $w \in \I_*$, since then it would hold that $w_j > w_jr_1 > w_jr_1r_2 > \dots > w_jr_1r_2\cdots r_p = w_{j+1}$
for all $j \in \NN$ where $w_j = w(r_1r_2\cdots r_p)^j$.
\end{algorithm}

\begin{remark}
Getting the outer loop to terminate is the bottleneck in all of our computations.
We should do this as soon as we can determine that $\sigma_i \neq w|_{V_J}$ for all $w \in \I_*$.
Checking whether $\cB_i$ is descent-periodic gives a computable sufficient condition for this to occur,
and this check is the only part of Algorithm~\ref{alg2} which relies implicitly on $(W,S)$ being a finite or affine Coxeter system. 
\end{remark}

Given $s,t \in S$ and a bijection $\theta : \{s,t\} \to \{s,t\}$, define 
$\cB_{s,t}^\theta =\cB_{s,t}^\theta = (\{\s, \t\}, \sigma, \varnothing)$ as the constant braid system 
in which
$\s = (\dots,s,t,s,t,s)$ and $\t = (\dots,t,s,t,s,t)$ are words with  length $m_\theta(s,t)$, and
 $\sigma$ is the linear map $V_{\{s,t\}} \to V_{\{s^*,t^*\}}\subset V$ with
$\sigma\alpha_s = \alpha_{\theta(s)^*}$ and $\sigma \alpha_t= \alpha_{\theta(t)^*}$.

\begin{algorithm}[Reduce spanning relations]\label{alg3}
Suppose we are given as inputs
distinct  $s,t \in S$.
The output of the algorithm is the forest of constant braid systems $\F_{s,t} = \F_{s,t}(W,S,*)$ 
whose roots are the two systems $\cB_{s,t}^\theta$, 
in which
the children of a vertex $\cB$ with domain $J$ are computed as follows:
\ben
\item[1.] If $\partial J = \varnothing$ then $\cB$ is a leaf.

\item[2.] Otherwise, 
apply Algorithm~\ref{alg1} to construct $ \cT(\cB\downarrow r)$ for each $r \in \partial J$.
Apply Algorithm~\ref{alg2} to each constant leaf in the resulting trees, and include any braid systems that are returned as children of $\cB$.
If there is no output for all of these systems, then $\cB$ is again a leaf.
\een
 The construction of $\F_{s,t}$ fails if any application of Algorithm~\ref{alg1} does not yield a finite tree
or any application of Algorithm~\ref{alg2} fails to terminate. 
\end{algorithm}

We say that $\F_{s,t}$ \emph{exists} 
if 
Algorithm~\ref{alg3} terminates successfully when given $s,t \in S$ as input.

\begin{example}
Suppose $(W,S)$ is the Coxeter system of type $A_3$ and $*=\id$ as in Example~\ref{ex-1}. 
Let $s=s_1$ and $t=s_2$.
The forest $\F_{s,t}$ exists and its two roots are the constant braid systems 
\[ 
\cB = 
\begin{cases}
\barr{ l | l }
\s &  (2,1) \\
\t &  (1,2) \\
\hline
\sigma &
\barr{llll}
 \alpha_1  \mapsto  \alpha_1 \\
\alpha_2  \mapsto  \alpha_2 
\earr
\\
\earr
\end{cases}
\qquand
\cB' = 
\begin{cases}
\barr{ l | l }
\s &  (2,1) \\
\t &  (1,2) \\
\hline
\sigma &
\barr{llll}
 \alpha_1  \mapsto  \alpha_2 \\
\alpha_2  \mapsto  \alpha_1
\earr
\\
\earr
\end{cases}
\]
We first compute $\cT(\cB\downarrow s_3)$. The root of the tree is
\[
\cB \downarrow s_3 
=
\begin{cases}
\barr{ l | l }
\s &  (2,1) \\
\t &  (1,2) \\
\hline
\sigma &
\barr{llll}
 \alpha_1  \mapsto  \alpha_1 \\
\alpha_2  \mapsto  \alpha_2 \\
\alpha_3  \mapsto  -x_1 \alpha_1 - x_2\alpha_2 - x_3 \alpha_3
\earr
\\
\hline
C & 
\barr{l} 
0\leq x_1 , x_2 , x_3  \\
0 = \tfrac{1}{2} + \tfrac{1}{2} x_1 - x_2 + \tfrac{1}{2} x_3 \\
0 = -x_1 + \tfrac{1}{2} x_2 \\ 
1 = |x_3|
\earr
\\
\earr
\end{cases}
\]
This system already has a unique solution, given by the ring homomorphism $\psi : \RR[x_1,x_2,x_3] \to \RR$
with $x_1 \mapsto \frac{2}{3}$, $x_2 \mapsto \frac{4}{3}$, and $x_3 \mapsto 1$.
Thus, $\cT(\cB\downarrow s_3)$ is  a path of length two with leaf
\[
(\cB \downarrow s_3)[\psi] =
\begin{cases}
\barr{ l | l }
\s &  (2,1) \\
\t &  (1,2) \\
\hline
\sigma &
\barr{llll}
 \alpha_1  \mapsto  \alpha_1 \\
\alpha_2  \mapsto  \alpha_2 \\
\alpha_3  \mapsto  -\tfrac{2}{3} \alpha_1 - \tfrac{4}{3}\alpha_2 -  \alpha_3
\earr
\earr
\end{cases}
\]
Executing Algorithm~\ref{alg2} with this braid system as input 
gives $\cB_0=(\cB \downarrow s_3)[\psi]$ and 
\[
\cB_1 = s_3\customstar \cB_0 =
\begin{cases}
\barr{ l | l }
\s &  (3,2,1) \\
\t &  (3,1,2) \\
\hline
\sigma &
\barr{llll}
 \alpha_1  \mapsto  \alpha_1 \\
\alpha_2  \mapsto   -\tfrac{2}{3} \alpha_1 - \tfrac{1}{3}\alpha_2 + \tfrac{2}{3}  \alpha_3 \\
\alpha_3  \mapsto  \tfrac{2}{3} \alpha_1 + \tfrac{4}{3}\alpha_2 +\tfrac{1}{3}  \alpha_3
\earr
\earr
\end{cases}
\]
which is invalid since $\sigma \alpha_2$ has both positive and negative coefficients. Hence Algorithm~\ref{alg2} terminates with no output, and $\cB$ has no children in $\F_{s,t}$.
On the hand, we have 
\[
\cB' \downarrow s_3 
=
\begin{cases}
\barr{ l | l }
\s &  (2,1) \\
\t &  (1,2) \\
\hline
\sigma &
\barr{llll}
 \alpha_1  \mapsto  \alpha_2 \\
\alpha_2  \mapsto  \alpha_1 \\
\alpha_3  \mapsto  -x_1 \alpha_1 - x_2\alpha_2 - x_3 \alpha_3
\earr
\\
\hline
C & 
\barr{l} 
0\leq x_1 , x_2 , x_3  \\
0 =  \tfrac{1}{2} x_1 - x_2 + \tfrac{1}{2} x_3 \\
0 = \tfrac{1}{2} -x_1 + \tfrac{1}{2} x_2 \\ 
1 = |x_3|
\earr
\\
\earr
\end{cases}
\]
which also has a unique solution, given by the ring homomorphism $\psi' : \RR[x_1,x_2,x_3] \to \RR$
with $x_1 \mapsto 1$, $x_2 \mapsto 1$, and $x_3 \mapsto 1$.
The tree $\cT(\cB'\downarrow s_3)$ is again a path of length two, now with leaf
\[
(\cB' \downarrow s_3)[\psi'] =
\begin{cases}
\barr{ l | l }
\s &  (2,1) \\
\t &  (1,2) \\
\hline
\sigma &
\barr{llll}
 \alpha_1  \mapsto  \alpha_2 \\
\alpha_2  \mapsto  \alpha_1 \\
\alpha_3  \mapsto  -\alpha_1 -\alpha_2 -  \alpha_3
\earr
\earr
\end{cases}
\]
Executing Algorithm~\ref{alg2} with this braid system as input 
gives $\cB_0=(\cB' \downarrow s_3)[\psi'] $, and then 
\[
\cB_1 = s_3\customstar \cB_0=
\begin{cases}
\barr{ l | l }
\s &  (3,2,1) \\
\t &  (3,1,2) \\
\hline
\sigma &
\barr{llll}
 \alpha_1  \mapsto  \alpha_2 + \alpha_3 \\
\alpha_2  \mapsto   -\alpha_2 \\
\alpha_3  \mapsto   \alpha_1 + \alpha_2
\earr
\earr
\end{cases}
\qquad
\cB_2 = s_3\bullet_{\mathrm{des}} \cB_1=
\begin{cases}
\barr{ l | l }
\s &  (2,3,2,1) \\
\t &  (2,3,1,2) \\
\hline
\sigma &
\barr{llll}
 \alpha_1  \mapsto   \alpha_3 \\
\alpha_2  \mapsto   \alpha_2 \\
\alpha_3  \mapsto   \alpha_1
\earr
\earr
\end{cases}
\]
which is valid and has no descents. Algorithm~\ref{alg2} therefore
terminates with $\cB_2$ as output. In $\F_{s,t}$,
this braid system is the unique child of $\cB'$ but has no children itself,
since its domain is $S$.
\end{example}

Define a braid system $\cB = (\{\s,\t\}, \sigma, C)$ with domain $J$ to be \emph{realizable} if for some $y,z \in \I_*$ there exists a solution $\psi$ with 
$(1\otimes_\RR \psi) \circ \sigma = y|_{V_J}$  and it holds that $\{\r\s,\r\t\}\subset \iR_*(z)$ for any (equivalently, every)  $\r \in \iR_*(y)$.
In this case, the pair $(y,z) \in \I_*\times \I_*$ \emph{realizes} the system.

\begin{lemma}\label{t-lem}
Suppose $(y,z) \in \I_*\times \I_*$ realizes a constant braid system
 $\cB$ with domain $J$.
Let $r \in \partial J$. Then $r \in \DesR(y)$ if and only if $(y,z)$ realizes $\cB \downarrow r$.
In this case, if the tree
 $\cT(\cB \downarrow r)$ is finite, then it has a leaf
which
is realized by $(x,z) \in \I_*\times \I_*$ for some $x \in \I_*$ with $x\leq y$.
\end{lemma}

\begin{proof}
The first assertion is evident. 
Let $r \in \partial J\cap \DesR(y)$, assume $\cT(\cB \downarrow r)$ is finite, and suppose $\cA=(\{\s,\t\}, \sigma, C)$ is a non-leaf vertex   realized by $(v,z)\in \I_*\times \I_*$.
It suffices to show that a child of $\cA$ in $\cT(\cB\downarrow r)$ 
is realized by $(u,z)$ for $u \in \I_*$ with $u \leq v$,
and this is straightforward.
\end{proof}

\begin{lemma}\label{3+-lem}
Suppose $(y,z) \in \I_* \times \I_*$ realizes a constant braid system $\cB $
with words $\s = (s_1,\dots,s_n)$ and $\t = (t_1,\dots,t_n)$. 
Let $\r=(r_1,\dots,r_k) \in \iR_*(y)$ and define $\sim$ as the relation
$ (r_1,\dots,r_k,s_1,\dots,s_n,\text{ --- }) \sim (r_1,\dots,r_k,t_1,\dots,t_n,\text{ --- })$.
 Assume $n>0$, and 
let $p= \ellhat_*(z) = k+n$
and  $q = m(s_n,t_n)$.
Suppose we execute Algorithm~\ref{alg2} with $\cB$ as input.
 If the algorithm terminates with output $\cB'$, then $\cB'$ is realized by $(x,z)$ for some $x\leq y $.
If the algorithm terminates with no output, then $\sAp{p-1} \cup \sAq{q+1} \cup \sB \Rightarrow {\sim}$.
\end{lemma}

\begin{proof} 
Let $\cB=\cB_0$, $\cB_1$, $\cB_2$, \dots be the sequence of braid systems defined in Algorithm~\ref{alg2} for the input $\cB$.
If $\cB_i$ is realized by $(x_i,z) \in \I_* \times \I_*$
and $r$ is the common first entry of the words of $\cB_{i+1}$,
then
$\cB_{i+1}$ is realized by $(x_{i+1},z)$
for the  element $x_{i+1}\in \I_*$ with $x_{i+1} < r^*\circ x_{i+1} \circ r = x_i$.
The first assertion follows from this observation.
Since each $\cB_i$ is thus realizable, it cannot occur that $\cB_i$ is invalid or descent-periodic.
Therefore, the only way  Algorithm~\ref{alg2} can terminate with no output is if some $\cB_i$ is redundant.
In this case, Lemma~\ref{3-lem} implies that for some $(r_1',\dots,r_k') \in \iR_*(y)$ we have $\sAp{p-1} \cup \sAq{q+1} \cup \sB \Rightarrow {\approx}$ where $\approx$ is the relation
$
 (r'_1,\dots,r'_k,s_1,\dots,s_n,\text{ --- }) \approx (r'_1,\dots,r'_k,t_1,\dots,t_n,\text{ --- }).
$
Since we have $\sAp{k} \cup \sB\Rightarrow {\sim_1}$ and $\sAp{k}\cup \sB\Rightarrow {\sim_2}$ for the relations 
$
  (r'_1,\dots,r'_k,s_1,\dots,s_n,\text{ --- }) \sim_1 (r_1,\dots,r_k,s_1,\dots,s_n,\text{ --- })
$
and
$
   (r'_1,\dots,r'_k,t_1,\dots,t_n,\text{ --- }) \sim_2 (r_1,\dots,r_k,t_1,\dots,t_n,\text{ --- })
$
by
Theorem~\ref{suff-thm},
and since it evidently holds that $\{ \approx, \sim_1, \sim_2\} \Rightarrow {\sim}$, it follows that
$\sAp{p-1} \cup \sAq{q+1} \cup \sB \Rightarrow {\sim}$ as desired.
\end{proof}

We refer to $(s_1,\dots,s_n,\text{ --- }) \sim (t_1,\dots,t_n,\text{ --- })$ 
as the word relation \emph{corresponding} to
a braid system 
$\cB = (\{\s,\t\}, \sigma, C)$ if $\s=(s_1,\dots,s_n)$ and $\t=(t_1,\dots,t_n)$.
Say that $\cB$ is \emph{trivial} if 
$\{\s,\t\} \subset \iR_*(z)$ for some $z \in \I_*$
and
$\sigma$ acts as the identity map on its domain.

\begin{definition}
When $s,t \in S$ are distinct and  $\F_{s,t}$ exists, 
we define $\sR_{s,t} = \sR_{s,t}(W,S,*) $ as the set of word relations corresponding to the trivial vertices in 
$\F_{s,t}$.
The elements of $\sR_{s,t}$ are all generalized half-braid relations for $(W,S,*)$,
which we refer to as the \emph{relations {induced} by $\F_{s,t}$}.
\end{definition}

Algorithms~\ref{alg1}, \ref{alg2}, and \ref{alg3} each depend on a fixed 
total order on $S$.
For each pair of generators $s,t \in S$ with $2<m(s,t)<\infty$,
choose such an order independently
and suppose Algorithm~\ref{alg3} (executed with respect to this choice)
succeeds in constructing $\F_{s,t}$, so that $\sR_{s,t}$ is defined.
Let $\isB=\isB(W,S,*)$ be the union of $\sB$ with the set of half-braid relations \eqref{folded-eq} of $(W,S,*)$.
Define   
$\sR=\sR(W,S,*)$ as the union of $\sR_{s,t}$ over all $s,t \in S$ with $2<m(s,t)<\infty$,
and let $\sR_{\min}$ be any minimal subset of $\sR$
with   $\isB \cup \sR_{\min} \Rightarrow   \sR$.
Note that $\isB \cap \sR_{\min} = \varnothing$.

\begin{theorem}\label{algwork-thm}
Assume the construction of $\sR_{\min}$ just given is defined. Each
set $\iR_*(w)$ for $w \in \I_*$ is then an equivalent class under the
 transitive relation generated by $\isB \cup \sR_{\min}$.
\end{theorem}

\begin{proof}
The relations in $\isB \cup \sR_{\min}$  preserve each set $\iR_*(w)$ for $w \in \I_*$ by definition.
Choose $r_1,\dots,r_k, s,t \in S$ and $m \in \PP$  such that the relation ${\sim}$
defined by \eqref{suff-eq} is contained in 
 $ \sA_{p,q}$
for $p=k+m$ and $q=m(s,t)$.
We may assume by induction that $\isB \cup \sR_{\min}\Rightarrow \sAp{p-1} \cup \sAq{q+1}$.
By Theorem~\ref{suff-thm}, it suffices to show  that $\isB \cup \sR_{\min}\Rightarrow {\sim}$.
If $k=0$ then ${\sim}$ is a half-braid relation in $\isB$.
On the other hand, if $k>0$ and $q=2$, then  Lemma~\ref{commute-impl-lem} implies that 
$\sAp{p-1} \cup \sAq{q+1} \cup \sB \Rightarrow{\sim}$.

Assume that $k>0$ and $2<q<\infty$.
By Proposition~\ref{hb-prop}
there are unique
elements $y,z \in \I_*$ such that $w\mapsto (ywy^{-1})^*$ is a bijection $\theta : \{s,t\} \to \{s,t\}$
and $(y,z)$ realizes the root vertex $\cB^\theta_{s,t}$ in $\F_{s,t}$.
We claim that either $\sAp{p-1} \cup \sAq{q+1}\cup \sB \Rightarrow {\sim}$
or 
$(1,z)$ realizes 
some descendant of $\cB^\theta_{s,t}$ in $\F_{s,t}$.
This alternative implies the theorem since in the first case it follows that
$\isB \cup \sR_{\min} \Rightarrow \sAp{p-1} \cup \sAq{q+1}\cup \isB \Rightarrow {\sim}$,
while in the second case   
$\sR$ contains a relation $\approx$ of the form
\be\label{proof-ref}
 (r'_1,\dots,r'_k,\underbrace{\dots,s,t,s}_{m\text{ terms}},\text{ --- }) \approx (r'_1,\dots,r'_k,\underbrace{\dots,t,s,t}_{m\text{ terms}},\text{ --- })
 \ee
 for some $(r_1',\dots,r_k') \in \iR_*(y)$, and 
when this occurs we can deduce that $\isB \cup \sR_{\min}\Rightarrow \isB \cup \sR \Rightarrow \{{\approx}\} \cup \sAp{k} \cup \sB  \Rightarrow{\sim}$ by the argument in the proof of Lemma~\ref{3+-lem}.

To prove the claim, let $y_0 = y<z$, $\cB_0=\cB_{s,t}^\theta$, and $J_0 =\{s,t\}$.
If $y_0=1$ then the claim holds immediately.  By construction $\DesR(y_0)$ is disjoint from $J_0$. 
If $y_0$ has a right decent  $r \in S - \partial J_0$, then
Lemma~\ref{relation-implication-lem} implies that
$\sAp{p-1} \cup \sAq{q+1}\cup \sB \Rightarrow {\sim}$.
If $y_0$ has a right descent $r \in \partial J_0$, then by Lemma~\ref{t-lem}
there exists $y_0' \in \I_*$ with $y_0' \leq y_0$
and a leaf vertex $\cB'_0$ in $\cT(\cB_0 \downarrow r)$ realized by $(y_0',z)$.
If Algorithm~\ref{alg2} applied to $\cB'_0$ has no output, 
then Lemma~\ref{3+-lem} implies that
for some $(r_1'\dots,r_k') \in \iR_*(y)$ we have $\sAp{p-1} \cup \sAq{q+1} \cup \isB \Rightarrow {\approx}$ where $\approx$ is the word relation \eqref{proof-ref},
in which case $ \isB \cup \isB^+ \Rightarrow \{{\sim}\}$.
If Algorithm~\ref{alg2} applied to $\cB'_0$ terminates with output $\cB_1$, then
$\cB_1$ is realized by $(y_1,z)$ for some twisted involution $y_1\leq y_0'$ by Lemma~\ref{3+-lem}.
When this occurs, define $J_1 = J_0 \cup \{r\}$ so that $\cB_1$ has domain $J_1$
and both words of $\cB_1$ involve only elements of $J_1$.
Then $\cB_1$ is a child of $\cB_0$ in $\F_{s,t}$ which is realized by $(y_1,z)$, and it necessarily holds that $y_1<y_0$.

By repeating the argument above, we deduce that either (1) $y_1=1$,  (2) $\sAp{p-1} \cup \sAq{q+1}\cup \isB \Rightarrow {\sim}$,
or (3) there exists a child $\cB_2$ of $\cB_1 $ in $\F_{s,t}$, with domain $J_2 \supset J_1$, whose words both involve only elements of $J_2$, and which is realized by $(y_2,z)$ for a twisted involution
$y_2 < y_1<y_0=y$.
Continuing in this way, we deduce by induction that our claim can only fail if there exists an infinite descending sequence of twisted involutions $y=y_0>y_1>y_2>\dots$, which is impossible.
\end{proof}

We have implemented Algorithm~\ref{alg3} in a few thousand lines of Python code \cite{code}.
This implementation successfully computes $\sR_{\min}$
whenever $(W,S)$ is a finite or affine Coxeter system of not too large rank. 
In type $\tilde E_8$, for example, our code calculates $\sR_{\min}$ 
in about twenty minutes.

\section{Reductions}
\label{red-sect}

Here, we discuss how one can efficiently reuse the output of Algorithm~\ref{alg3}.
A \emph{morphism} of twisted Coxeter systems $\phi : (W,S,*) \to (W',S',\diamond)$ is
a group homomorphism $\phi : W \to W'$ with $\phi(S) \subset \{1\} \cup S'$ and $\phi(w^*) = \phi(w)^\diamond$ for all $w \in W$.
If a morphism $\phi$ is injective,
then it 
 induces a natural map $\sA(W,S,*) \to \sA(W',S',\diamond)$
between the corresponding sets of generalized half-braid relations \eqref{gen-rel-eq}: this is defined by applying $\phi$ to 
the first $k+m$ terms in a given relation \eqref{suff-eq}.
Given  $J\subset  S$, write  $\partial_K J$
for the set of $s \in K-J$ such that $m(s,t) > 2$ for some $t \in J$.

\begin{definition}
A \emph{bounded Coxeter system} is a quadruple $(W,S,J,*)$ where $(W,S,*)$ is a twisted Coxeter system,
 $J=J^*\subset S$, and $ \partial_S J \cup J  \neq S$.
We say that 
$ \phi : (W,S,J,*) \to (W',S',\diamond)$
is a \emph{bounded embedding} 
 if $\phi$ is
an injective morphism   $\phi : (W,S,*) \to (W',S',\diamond)$
and
$\partial_{S'} \phi(J) = \phi\( \partial_S J\)$.
\end{definition}

Suppose  $(W,S,J,*)$
 is a bounded Coxeter system and $s,t \in S$. Let $V$ be the geometric representation of $(W,S)$ and
 fix a total order on $S$.
We call $(W,S,J,* \hs|\hs s,t)$ a \emph{parabolic configuration} if 
 Algorithm~\ref{alg3} succeeds in constructing  $\F_{s,t}=\F_{s,t}( W, S,*)$,
and for each vertex $\cB=(\{\s,\t\},\sigma,\varnothing)$ in $\F_{s,t}$ it holds that $(\partial_S K \cup K)\subset J$ and $\sigma(V_K)\subset V_J$ where $K$ is the domain of $\cB$.

\begin{theorem}\label{bounded-thm}
Let  $\phi : S \to S'$ be an order-preserving map between totally ordered finite sets,
choose distinct $s,t \in S$, and set $s'= \phi(s)$ and $t'=\phi(t)$. Suppose
$(W',S',\diamond)$ is a twisted Coxeter system,
 $(W,S,J,* \hs|\hs s,t)$ 
is a parabolic configuration,
and
$\phi$ extends to a bounded embedding $ (W,S,J,*) \to (W',S',\diamond)$. 
The forests $\F_{s,t}(W,S,*)$ and $\F_{s',t'}(W',S',\diamond)$  then both exist,
and $\phi$ induces a bijection between the sets of 
 generalized half-braid relations $\sR_{s,t}(W,S,*)$ and $\sR_{s',t'}(W',S',\diamond)$  which each induces.
Moreover, $(W',S',\phi(J),\diamond\hs|\hs,s',t')$ is also a parabolic system.
\end{theorem}

\begin{proof}
We may assume   that $S \subset S'$, $W= W'_S$, $* = \diamond|_W$, and that $\phi$ is the inclusion map $W\hookrightarrow W'$, so $s=s'$ and $t=t'$ and $\partial_S J = \partial_{S'}J$.
Let $V$ be the geometric representation of $(W',S')$. 
By definition,  $\F= \F_{s,t}(W,S,*)$ exists and   $K\subset J$ and $\sigma(V_K)\subset V_J$ 
for each vertex $\cB=(\{\s,\t\},\sigma,\varnothing)$ with domain $K$. 
Each vertex in $\F$ may be regarded as braid system relative to  $(W',S',\diamond)$,
and it suffices to show that under this identification $\F= \F_{s',t'}(W',S',\diamond)$.

We first examine Algorithm~\ref{alg1}.
Suppose $\cB_0$ is a non-leaf vertex of $\F$ with domain $K$. Note that $ \partial_S K = \partial_{S'} K  \subset J$.
Regard $\cB_0$ as a braid system relative to $(W,S,*)$,
and let $\cB'_0$ consist of the same data, but considered as a braid system relative to $(W',S',\diamond)$.
Choose $r \in \partial_S K$
and define braid systems $\cB=\cB_0\downarrow r $ and $\cB' =\cB'_0\downarrow r$
for $(W,S,*)$ and $(W',S',\diamond)$.
Let $\cT = \cT(\cB)$ and
$\cT'=\cT(\cB')$ be the outputs  of Algorithm~\ref{alg1} applied to $\cB$ and $\cB'$.
Fix $u \in S - (\partial_S J \cup J)$, 
and let $\cA =(\{\s,\t\}, \sigma, C)$ and $\cA' =(\{\s',\t'\}, \sigma', C')$ be vertices of $\cT$ and $\cT'$.
The following properties hold by induction,
since all elements of $\{r\} \cup K \subset J$ and $S'-(\partial_S J\cup J)$ commute with each other:
\ben
\item[(1)]
Let 
$U = V_{S-\{u\}} \otimes_\RR \RR[S-\{u\}]$ and define 
$ \Upsilon = \alpha_u \otimes_\RR x_u$
and 
$\Upsilon' = \alpha_u \otimes_\RR x_u + \sum_{v \in S' - S} \alpha_{v}\otimes_\RR x_{v}$.
 The images of $\sigma$ and $\sigma'$ are contained in the sets
 $U \oplus \RR\Upsilon$  and $ U \oplus \RR\Upsilon'$.

\item[(2)] The set $C$ may be partitioned into a set of constraints which do not involve the variable $x_u$ at all,
and a set of constraints which specify that this variable is zero, nonnegative, and/or nonpositive.
The set $C'$ may likewise be partitioned into a set of constraints which do not involve any of the variables $x_u$ or $x_{v}$ for $v \in S'-S$,
and a set of constraints which specify that all of these variables are simultaneously zero, nonnegative, and/or nonpositive.
\een
These observations make the following operation well-defined.
Let $\varepsilon : U \oplus \RR \Upsilon \to U\oplus \RR \Upsilon'$ be the linear map which is the identity on $U$ and which maps $\Upsilon \mapsto \Upsilon'$.
For each vertex $\cA= (\{\s,\t\}, \sigma, C) $ in $ \cT$,
define $\lambda(\cA)$ as the braid system given by 
replacing $\sigma$ with $\varepsilon \sigma$, and then adding to each constraint in $C$ of the form $x_u =0$ or $x_u \leq 0$ or $x_u \geq 0$  the analogous identity for  $x_{v}$ for all $v  \in S'-S$.
Regard $\lambda(\cA)$ as a braid system for $(W',S',\diamond)$ and define $\lambda(\cT)$ as the tree given by applying $\lambda$ to each vertex of $\cT$.
It holds by definition that $\lambda(\cB) = \cB'$. Using properties (1) and (2), it is straightforward to deduce 
 by induction that in fact
 $\lambda(\cT) = \cT'$.
 
%

The leaf vertices in $\cT  = \cT(\cB_0 \downarrow r)$ which are not invalid or redundant may be trivially identified with the leaf vertices in $\lambda(\cT)$ which are not invalid or redundant. 
Since Algorithm~\ref{alg2} has no dependence on the ambient Coxeter system, 
 the children of $\cB_0$ in $\F$ may be identified with the children of $\cB_0'$ in
 $\F_{s,t}(W',S',\diamond)$.
By induction,
we have $\F=\F_{s,t}(W',S',\diamond)$ as needed.
\end{proof}

\begin{table}[h]
\begin{center}
   \begin{tabular}{|c | c | c | c |}
    \hline
    Type &  Coxeter Diagram &     Type &  Coxeter Diagram  \\ 
     \hline
   $ A_n$ &  \andiagram  &    $\tilde A_n$ &  \atildediagram \\
        \hline 
   $ B_n$ &  \bndiagram  &    $\tilde B_n$ &  \btildediagram \\
        \hline
$ C_n$ &  \bndiagram &      $\tilde C_n$ &  \ctildediagram    \\
        \hline 
   $ D_n$ &  \dndiagram  &    $\tilde D_n$ &  \dtildediagram        \\         \hline     
  \end{tabular}
\end{center}
\caption{Coxeter systems of classical type. Single (double) edges have weight three (four).}\label{affine-tbl}
\end{table}

\begin{example}\label{atilde-ex}
Suppose $(\tilde W,\tilde S)$ is the Coxeter system of type $\tilde A_n$ ($n\geq 2$) with Coxeter diagram
as in Table~\ref{affine-tbl},
so that $\tilde S = \{s_0,s_1,\dots, s_n\}$. 
Let $(W,S)$ be the Coxeter system of type $A_8$ and define $J = \{s_2,s_3,\dots,s_7\}$ and $(s,t)=(s_4,s_5)$.
We have checked that $(W,S,J,\id \hs|\hs s,t)$ is a parabolic configuration using \cite{code}.
The corresponding forest $\F_{s,t}$ induces no word relations  independent of the half-braid relations in type $A_8$.
When $n>8$,
there exists a bounded embedding $\phi : (W,S,J,\id ) \to (\tilde W, \tilde S, \id) $
with $\phi(s) =  s'$ and $\phi(t) =  t'$
for any $ s',  t' \in \tilde S$  with $m( s',  t') > 2$.
It follows by Theorems~\ref{algwork-thm} and \ref{bounded-thm}
that  $(\tilde W, \tilde S, \id)$  is \emph{perfectly braided} in the sense of Corollary~\ref{main-cor} when $n>8$.
When $n\leq 8$, the same conclusion can be checked directly using \cite{code}. 
\end{example}

Let $A_n$ ($n\geq 1$), $B_n$ $(n\geq 2)$, and $D_n$ ($n\geq 4$) 
refer to the Coxeter systems $(W,S)$ in which $S=\{s_1,s_2,\dots, s_n\}$
and the corresponding Coxeter diagrams are as in Table~\ref{affine-tbl}.
When $(W,S)$ is a Coxeter system with $S = \{s_1,s_2,\dots,s_n\}$, 
 let $(W,S)\times (W,S)$ denote the Coxeter system
whose  simple generators 
are $\{ s_1,s_2\dots, s_n\} \sqcup \{ \tilde s_1,\tilde s_2,\dots, \tilde s_n\}$,
in which   $m(s_i, s_j) = m(\tilde s_i, \tilde s_j)$ and $m(s_i, \tilde s_j) = 2$ for all $i,j \in [n]$.
We write $^2(W\times W)$ to denote the twisted Coxeter system
given by $(W,S)\times (W,S)$ with the involution $*$
satisfying  $s_i^* = \tilde s_i$ for $i \in [n]$.
When $(W,S)$
has type $A_n$, $B_n$, or $D_n$, we write
$^2(A_n \times A_n)$, $^2(B_n\times B_n)$, and $^2(D_n\times D_n)$ in place of $^2(W\times W)$.
Order the generators of the product system 
as $s_1 < s_2<\dots <s_n < \tilde s_1<\tilde s_2 < \dots < \tilde s_n$.
The following describes several common parabolic configurations. 
\ben
\item Define the configuration of type $A_n$ ($n\geq 5$), $B_n$ ($n\geq 5$),
and $D_n$ ($n \geq 7$)
 to be $(W,S,J,*\hs|\hs s,t)$ where $(W,S)$
is the Coxeter system of type $A_n$, $B_n$ or $D_n$, respectively; 
 $*=\id$; $\{s,t\} = \{s_4,s_5\}$;
and
$J$ is either $\{ s_2,s_3,s_4,s_5\} $ in types $A_5$ and $B_5$,
$\{ s_2,s_3,s_4,s_5,s_6\}$ in types $A_6$ and $B_6$,
$\{ s_2,s_3,s_4,s_5,s_6,s_7,s_8\} $ in type $D_8$,
or
$\{ s_2,s_3,s_4,s_5,s_6,s_7\} $ in the other cases.

\item[2.] Define the configurations of types $^2A_n$ ($n\geq 8$) and $^2D_n$ ($n\geq 7$)
exactly as in types $A_n$ and $D_n$,
but let $*$ be the unique nontrivial automorphism of the corresponding Coxeter diagram.

\item[3.] 
Let $X$ be the letter $A$, $B$, or $D$. Fix $n\in \PP$  such that the parabolic configuration  
$(W,S,J,\id\hs|\hs s,t)$ of type $X_n$
is defined. 
The configuration of type $^2(X_n\times X_n)$ is then $(W',S',J',\diamond,\hs|\hs s,t)$
where 
 $(W',S', \diamond)$ has type $^2(X_n\times X_n)$,
 $J' = J \sqcup J^\diamond$, and $\{s,t\} = \{s_4,s_5\}$.
\item[4.]
Define the configurations of types $D_7'$, $^2D_7'$, and $^2(D_7\times D_7)'$
in the same way as $D_7$, $^2D_7$, and $^2(D_7\times D_7)$,
but in each case let $\{s,t\}$ be $\{s_5,s_6\}$ instead of $\{s_4,s_5\}$.
\een

\begin{proposition}\label{par-prop}
Relative to the implementation \cite{code} of Algorithm~\ref{alg3},
the configurations of types 
$A_n$, 
$B_n$, 
$D_n$,  
$^2A_n$, 
$^2D_n$, 
$^2(A_n\times A_n)$, 
$^2(B_n \times B_n)$, 
$^2(D_n\times D_n)$, 
$D_7'$, $^2D_7'$, and $^2(D_7\times D_7)'$ defined above are all parabolic.
Suppose $(W,S,J,*\hs|\hs s,t)$ is one of these parabolic configurations,
and consider the word relations $\sR_{s,t}$ induced by the corresponding forest $\F_{s,t}$.
In types $^2A_9$, $B_5$, $D_7$, $D_7'$, $^2D_7$, and $^2D_7'$, 
there exists a single relation ${\sim} \in \sR_{s,t}$ such that $\isB \cup \{\sim\} \Rightarrow \sR_{s,t}$;
Table~\ref{parabolic-table} shows possible values for this relation.
 In all of the other types, it holds that $\isB \Rightarrow \sR_{s,t}$.
In either case, we have $ \isB \cup \isB^+ \Rightarrow \sR_{s,t}$
with $\isB^+ = \isB^+(W,S,*)$  as in Theorem~\ref{main-thm}.
\end{proposition}

\begin{proof}
We have checked (a) and (b) by computer \cite{code}
in type $^2A_n$ when $n \leq 16$ and in the other types when $n\leq 8$. In any other type the desired result follows by Theorem~\ref{bounded-thm}.
The last assertion follows by checking that each relation ${\sim}$ in Table~\ref{parabolic-table}
is implied by $\isB \cup \isB^+$, which is straightforward.
\end{proof}

\begin{table}
\begin{center}
  \begin{tabular}{| c | c |}
    \hline
    Type &  Relation \\ \hline
   $^2A_{9}$ &  $(5,6,4,5,\text{ --- }) \sim (5,6,5,4,\text{ --- })$  \\       
   $B_5$ &  $(5,4,3,4,5,4,\text{ --- }) \sim (5,4,3,5,4,5,\text{ --- })$  \\       
    $D_7$ &   $(6, 4, 5, 7, 5, 6, 4, 5,\text{ --- }) \sim (6, 4, 5, 7, 5, 6, 5, 4,\text{ --- })$  \\       
    $D_7'$ &  $(6, 4, 5, 7, 5, 4, 5, 6,\text{ --- }) \sim (6, 4, 5, 7, 5, 4, 6, 5,\text{ --- })$  \\    
     $^2D_7$ &  $(6, 5, 4, 3, 7, 4, 5, 6, 5, 7, 4, 5,\text{ --- }) \sim (6, 5, 4, 3, 7, 4, 5, 6, 5, 7, 5, 4,\text{ --- })$   \\       
    $^2D_7'$ &   $(5, 7, 5, 6,\text{ --- }) \sim  (5, 7, 6, 5,\text{ --- })$  \\  \hline     
  \end{tabular}
\end{center}

\caption{Values for the word relation ${\sim} \in \sR_{s,t}$ in Proposition~\ref{par-prop}.}\label{parabolic-table}
\end{table}

\section{Proof of Theorem~\ref{main-thm}}\label{app-sect}

A twisted Coxeter system $(W,S,*)$
 is \emph{irreducible} if $*$ acts transitively on the connected components of its Coxeter 
diagram $\Gamma$.
If $J=K\cup K^*$ where $K\subset S$ is the set of vertices in a
connected component of $\Gamma$, then $(W_J, J, *)$ is an \emph{irreducible component} of $(W,S,*)$.
The group $W$ is the direct product of the parabolic subgroups corresponding to its irreducible components.
If $J =J^*\subset S$,
then every word in $\iR_*(z)$ for $z \in \I_* (W_J)= \I_* \cap W_J$ 
is a sequence of elements of $J$ \cite[Proposition 2.10]{HMP2}.
It is straightforward to deduce from this property that Theorem~\ref{main-thm}
holds if and only the result holds for each irreducible twisted Coxeter system of finite or affine type.
If $(W,S,*)$ is irreducible but $\Gamma$ is not connected, Theorem~\ref{main-thm} reduces to the following:

\begin{proposition}\label{wxw-prop}
A twisted Coxeter system of type $^2(W\times W)$ is perfectly braided.
\end{proposition}

\begin{proof}
The result follows as a simple exercise; the details are left to the reader.
\end{proof}

Similarly, 
when Theorem~\ref{main-thm} holds for $(W,S,*)$, it follows that the result also holds for $(W_J, J, *)$ for any $J=J^*\subset S$.
Since Theorem~\ref{suff-thm} implies Theorem~\ref{main-thm} in rank two,
we therefore just need to prove Theorem~\ref{main-thm}
when $(W,S)$ is of type  $\tilde A_n$ ($n\geq 2$), $\tilde B_n$ ($n\geq 3$), $\tilde C_n$ ($n\geq 2$), $\tilde D_n$ ($n\geq 4$), $\tilde E_6$, $\tilde E_7$, $\tilde E_8$, $F_4$, $\tilde F_4$, $\tilde G_2$, or $H_4$. 
 We refer to \cite[Chapter 2]{Humphreys} for the definitions of these Coxeter systems.
Our proof of Theorem~\ref{main-thm} reduces to two propositions:

\begin{proposition}
Theorem~\ref{main-thm} holds for any twisted Coxeter system $(W,S,*)$ in which $(W,S)$ is an irreducible Coxeter system of type $\tilde E_6$, $\tilde E_7$, $\tilde E_8$, $F_4$, $\tilde F_4$, $\tilde G_2$, or $H_4$.
\end{proposition}

\begin{proof}
Using \cite{code}, we have calculated a set of relations spanning the sets $\iR_*(z)$ in each type.
On a basic laptop these computations take between 15 seconds (in type $\tilde G_2$) and  20 minutes (in type $\tilde E_8$).
The resulting relations in each type  are easily seen to be equivalent to the desired set $\isB \cup \isB^+$.
\end{proof}

\begin{proposition}
Theorem~\ref{main-thm} holds for any twisted Coxeter system $(W,S,*)$ in which $(W,S)$ is an irreducible affine Coxeter system of classical type.
\end{proposition}

We say that a twisted Coxeter system $(W',S',\diamond)$ is \emph{covered}
by a collection of parabolic configurations $\sP$ if for each $s',t' \in S'$ with $2<m(s',t')<\infty$
there exists $(W,S,J,*\hs|\hs s,t) \in \sP$ and a bounded embedding
$\phi :(W,S,J,*\hs|\hs s,t) \to (W',S',\diamond)$ with $\phi(s) = s'$ and $\phi(t) =t'$.

\begin{proof}
We may assume that $(W,S)$ has type $\tilde A_n$, $\tilde B_n$, $\tilde C_n$, or $\tilde D_n$ and $n\geq 2$.
Label the elements of $S$ as  in Table~\ref{affine-tbl}.
For each type,
we can consider just one involution $*$ from each conjugacy class in the group of $S$-preserving automorphisms of $(W,S)$. In type $\tilde A_n$ there are four cases:
\ben
\item[(A1)] Suppose $(W,S)$ has type $\tilde A_n$ and $*=\id$.  This case
was already discussed in Example~\ref{atilde-ex}, where we saw that $(W,S,*)$ is perfectly braided, as Theorem~\ref{main-thm} predicts.

\item[(A2)] Suppose $(W,S)$ has type $\tilde A_{n}$  and $s_i^*= s_{n-i}$ for all $i$. 
If $n>9$, then $(W,S,*)$ is covered by the parabolic configurations of type $^2A_8$, $^2A_9$, $^2A_{10}$, \dots, $^2A_{16}$, and $^2(A_8\times A_8)$. 

\item[(A3)] Suppose $(W,S)$ has type $\tilde A_{2n-1}$  and $s_i^* = s_{2n-i}$ for $i \in [n]$ and $s_0^*= s_0$. 
If $n>8$, then $(W,S,*)$ is covered by the parabolic configurations of type
$^2A_9$, $^2A_{11}$, $^2A_{13}$, $^2A_{15}$, and $^2(A_8\times A_8)$. 

\item[(A4)] Suppose $(W,S)$ has type $\tilde A_{2n-1}$ and  $ s_i^*= s_{i+n}$ for $0\leq i < n$.
If $n>8$ then $(W,S,*)$ is covered by the single parabolic configuration of type
$^2(A_8\times A_8)$. 
\een
In each case,  if $n$ is sufficiently large then $(W,S,*)$ is covered by a finite set of parabolic configurations,
and  the desired result follows by combining Theorems~\ref{algwork-thm} and  \ref{bounded-thm} and Proposition~\ref{par-prop}.
For the remaining values of $n$ we have checked Theorem~\ref{main-thm} directly by computer, using our implementation \cite{code} of Algorithm~\ref{alg3}. 
We conclude that Theorem~\ref{main-thm} holds in type $\tilde A_n$.
Our analysis in types $\tilde B_n$, $\tilde C_n$, and $\tilde D_n$ is similar. There are two cases for each type:
\ben
\item[(B1)] Suppose $(W,S)$ has type $\tilde B_n$ and $*=\id$.
If $n>10$ then $(W,S,*)$ is covered by the parabolic configurations of type $A_8$, $B_5$, $B_6$, $B_7$, $B_8$, $D_7$, $D_7'$, $D_8$, and $D_9$.

\item[(B2)] Suppose $(W,S)$ has type $\tilde B_{n}$ and $* $ is nontrivial.
If $n>10$ then $(W,S,*)$ is covered by the parabolic configurations of type $A_8$, $B_5$, $B_6$, $B_7$, $B_8$, $^2D_7$, $^2D_7'$, $^2D_8$, and $^2D_9$.

\item[(C1)] Suppose $(W,S)$ has type $\tilde C_n$ and $*=\id$.
If $n>9$ then $(W,S,*)$ is covered by the parabolic configurations of type $A_8$, $B_5$, $B_6$, $B_7$, and $B_8$.

\item[(C2)] Suppose $(W,S)$ has type $\tilde C_{n}$ and $*$ is nontrivial.
If $n>18$ then $(W,S,*)$ is covered by the parabolic configurations of type $^2(A_8\times A_8)$, 
$^2(B_5 \times B_5)$, $^2(B_6 \times B_6)$, $^2(B_7 \times B_8)$, 
 $^2(B_8 \times B_8)$, 
and either
$^2A_8$, $^2A_{10}$, \dots, $^2A_{16}$ (if $n$ is odd)
or 
$^2A_9$, $^2A_{11}$, \dots, $^2A_{15}$ (if $n$ is even).

\item[(D1)] Suppose $(W,S)$ has type $\tilde D_n$ and $*$ is a diagram involution which fixes $s_2,s_3,\dots,s_{n-2}$.
There are four choices for $*$.
If $n>11$, then $(W,S,*)$ is covered by (a subset of) the parabolic configurations of type $A_8$, $D_7$, $D_7'$, $D_8$, $D_9$, 
$^2D_5$, $^2D_6$, $^2D_7$, $^2D_7'$, $^2D_8$, and $^2D_9$.

\item[(D2)] Suppose $(W,S)$ has type $\tilde D_{n}$ and $s_i^*= s_{n-i}$
for $0\leq i\leq n$.
If $n>20$ then $(W,S,*)$ is covered by 
the parabolic configurations of type
$^2(A_8\times A_8)$, $^2(D_7 \times D_7)$, $^2(D_7 \times D_7)'$, $^2(D_8 \times D_8)$,
$^2(D_9 \times D_9)$, and either $^2A_8$, $^2A_{10}$, \dots, $^2A_{16}$ (if $n$ is even)
or 
$^2A_9$, $^2A_{11}$, \dots, $^2A_{15}$ (if $n$ is odd).
\een
In all of these cases, we obtain the desired result either from Theorems~\ref{algwork-thm} and  \ref{bounded-thm} and  Proposition~\ref{par-prop} when $n$ is sufficiently large, or by direct computation \cite{code}.
\end{proof}

Together, the results in this section imply Theorem~\ref{main-thm}.

\end{document}